\documentclass[12pt, a4paper]{amsart}
\usepackage[T2A]{fontenc}
\usepackage[utf8]{inputenc}
\usepackage{amsmath}
\usepackage{extarrows}
\usepackage{amssymb,amsthm,amscd}
\usepackage{hyperref}
\usepackage{graphicx}
\usepackage{fullpage}

\usepackage[shortlabels]{enumitem}
\usepackage{multirow}

\usepackage{tikz}
\usetikzlibrary{calc}

\usepackage[obeyDraft]{todonotes}

\theoremstyle{plain}
\newtheorem{theorem}{Theorem}[section]
\newtheorem{lemma}[theorem]{Lemma}

\newtheorem{proposition}[theorem]{Proposition}
\newtheorem{corollary}[theorem]{Corollary}
\theoremstyle{definition}
\newtheorem{definition}[theorem]{Definition}
\newtheorem{example}[theorem]{Example}
\newtheorem{remark}[theorem]{Remark}

\newtheorem{notation}[theorem]{Notation}

\theoremstyle{remark}
\newtheorem{claim}{Claim}[theorem]
\newenvironment{claimproof}
  {\begin{proof}[Proof of the claim]}
  {\end{proof}}

\def\Aut{\operatorname{Aut}}

\def\Bir{\operatorname{Bir}}

\def\supp{\operatorname{supp}}

\def\Spec{\operatorname{Spec}}

\def\PGL{\operatorname{PGL}}

\def\PSL{\operatorname{PSL}}

\def\id{{\mathrm{id}}}
\def\reg{{\rm reg}}
\def\sing{{\rm sing}}

\def\PP{{\mathbb P}}
\def\ZZ{{\mathbb Z}}

\def\RR{{\mathbb R}}
\def\QQ{{\mathbb Q}}

\def\FF{{\mathbb F}}
\def\AA{{\mathbb A}}

\def\KK{{\mathbb K}}

\def\O{{\mathcal O}}

\def\cO{{\mathcal O}}
\def\0{\circ}

\title{Open surfaces with a triangle at infinity}

\author[D. Chunaev]{Dmitriy Chunaev}
\email{dchunaev@hse.ru}
\address{Lomonosov Moscow State University, Moscow, Russia
\newline\indent HSE University, Faculty of Computer Science, Moscow, Russia}

\author[A. Perepechko]{Alexander Perepechko} 
\email{a@perep.ru}
\address{
HSE University, Faculty of Computer Science,
Pokrovsky blvd. 11, Moscow, 109028 Russia} 

\author[D. Shunin]{Daniil Shunin}
\email{sshunindaniil@gmail.com}
\address{Lomonosov Moscow State University, Moscow, Russia
\newline\indent HSE University, Faculty of Computer Science, Moscow, Russia}

\thanks{The work was supported by the Theoretical Physics and Mathematics Advancement Foundation ``BASIS''}



\begin{document}

\begin{abstract}
  A triangle surface is an open algebraic surface completed by a triangle of  contractible $(-1)$-curves.
We establish a combinatorial description of their completions.
We also show that affine triangle surfaces are exactly cubic surfaces of Markov type. 
Explicit description of their automorphism groups is provided.
\end{abstract}

\maketitle

\section{Introduction}

The Markov equation $x^2+y^2+z^2=3xyz$ has been studied since 1879, see~\cite{Mar1879, Mar1880}; for surveys we refer to~\cite{Aig13, Sil26}.
Its positive integer solutions, the \emph{Markov triples}, are permuted by the three involutions of the form $(x,y,z)\mapsto(3yz-x,\,y,\,z)$.
These involutions organize the Markov triples into an infinite tree rooted at the solution $(1,1,1)$; the classical relation of this tree to the modular group goes back to Cohn~\cite{Coh55}.
The automorphism groups of wide classes of affine algebraic surfaces exhibit a similar behaviour.

Let $\KK$ be an algebraically closed field of characteristic zero.
The well-known \emph{Markov surface} is the affine surface in $\AA^3=\Spec\KK[x,y,z]$ defined by the rescaled Markov equation $xyz=x^2+y^2+z^2$.
There are many ways to generalize it, e.g., see~\cite{Bar91, BaSe24, DeSt24, GyMa23, HaTe20, UlYi22}.
We briefly recall some directions of generalization.

In \cite{Lam16,Gyo21,BaSe24,GyMa23}, a generalized equation $x^2+y^2+z^2+xy+yz+xz=6xyz$ was studied, as well as its variations, in the context of mutations of cluster algebras.
Integer solutions to a generalized equation $Ax^2+By^2+Cz^2-Dxyz=0$ were studied in \cite{Mor53,Ros79,JinSch01,HaTe20}.
Holomorphic automorphisms of Markov-type surfaces 
\begin{equation}\label{eq:intro-markov-type}
  x^2+y^2+z^2+xyz=ax+by+cz+d 
\end{equation}
are studied in \cite{ReRo24,Abb24-preprint,And25}.
And the dynamics of $\PGL(2,\ZZ)$ acting on such surfaces are studied in \cite{CaLo09}.

An algebro-geometric view is also considered as a primary one in the following papers.
An influential paper \cite{ElHu74} studies cubic surfaces of Markov type with a triangle of negative curves at the boundary. 
We mostly follow this approach.
Connection between involutions of cubic surfaces and singularities at the boundary are well studied in \cite{Mor52, Kol24-irred,KoVi25}.
For the classical theory of cubic surfaces, we refer to~\cite{Seg42, Dolg_CAG12}.
The Fricke surfaces go back to~\cite{FrKl12}; the so-called double Fricke surface $(x+y+z)^2=9xyz$ is studied in~\cite{UlYi22}.
 Finally, the \emph{Markov-like surfaces}, given in $\AA^3$ by the equation $x^2+y^2+z^2-xyz=c$, were studied in~\cite{Pe21}, and their automorphism groups were completely described.

All mentioned generalizations are defined by an equation of the form $xyz=f(x,y,z)$ with $\deg f\leq2$; they all share the property that the boundary divisor $D$ of the completion in $\PP^3$ consists of three (-1)-curves with normal crossings.
We study normal affine surfaces with such property. 
More precisely, we call a normal affine surface $Y$ a \emph{triangle surface} if it admits a completion whose boundary divisor is a triangle of contractible $(-1)$-curves lying in the regular locus.
The criterion~\cite[Theorem~4.17]{PeZa26} allows to describe completely transformations of the dual graph of components of $D$ corresponding to the birational transformations of the completion.
In this paper, we use this criterion to study the automorphism groups of triangle surfaces.

We show that this purely geometric definition of a triangle surface implies an embedding into $\AA^3$ with the equation \eqref{eq:intro-markov-type}.
Thus, the triangle surfaces are exactly the cubic surfaces of Markov type.

We extend the theory of triangle surfaces to the case of a not necessarily affine surface $Y$ admitting a completion by a triangle of $(-1)$-curves.
We call such completions \emph{triangle completions} and the arising completions by a cycle of two $0$-curves \emph{$(0,0)$-completions}.
The combinatorics of all such completions of $Y$ is expressed in a simplicial complex, which we call the \emph{triangle complex} $T(Y)$, and describe in Section~\ref{sec:triangle-completions}.
Its triangles correspond to triangle completions of $Y$, its edges correspond to $(0,0)$-completions, and its vertices correspond to the boundary components of these completions.
Moreover, vertices of $T(Y)$ are in natural bijection with the fibrations of $Y$ over $\AA^1$ with general fiber $\AA^1\setminus\{0\}$, see Section~\ref{sec:fibrations}.

In Theorem~\ref{th:aut-action}, there is a homomorphism $\Aut(Y)\to\PGL_2(\ZZ)$, whose kernel consists of the automorphisms of $Y$ preserving every component and every node of the boundary divisor of a triangle completion.
For a triangle surface, the kernel consists of the sign changes preserving the equation and is finite of order at most four, see Remark~\ref{cr:aut-natural-cubic}(1).
In fact, for triangle surfaces the description of the automorphism group becomes completely explicit.
By El-Huti~\cite{ElHu74}, the three Vieta involutions $\sigma_x,\sigma_y,\sigma_z$ of the cubic~\eqref{eq: general_equation} generate their free product $G_\sigma=\langle\sigma_x\rangle*\langle\sigma_y\rangle*\langle\sigma_z\rangle$, and $\Aut(Y)$ is generated by $G_\sigma$ together with the linear automorphisms of $Y$, see Proposition~\ref{prop: generators_El-Huti}. 
Building on this, Theorem~\ref{thm: automorphisms_of_cubic_surfaces} computes $\Aut(Y)$ for every triangle surface: it is the semidirect product of $G_\sigma$ by the group $\mathrm{Lin}(Y)$ of linear automorphisms of $Y$ which is a finite subgroup of $S_4$ defined by the coefficients $a$, $b$, $c$ of $\eqref{eq:intro-markov-type}$.
The largest group, $K_4\rtimes\PGL_2(\ZZ)$, occurs exactly for $a=b=c=0$, that is, for the Markov-like surfaces $xyz=x^2+y^2+z^2+d$, among them the Markov surface itself; this recovers the description obtained in~\cite{Pe21}.

The paper is organized as follows.
Section~\ref{sec:prelim} recalls the terminology of~\cite{PeZa26} concerning completions by normal crossing divisors.
In Section~\ref{sec:triangle-completions} we introduce the triangle complex $T(Y)$ and describe the action of $\Aut(Y)$ on it.
Section~\ref{sec:fibrations} attaches a fibration to every vertex of $T(Y)$ and characterizes these fibrations intrinsically.
In Section~\ref{sec:triangle-surfaces} we show that every triangle surface is a cubic surface $xyz=f(x,y,z)$ with $\deg f\leq2$; in Section~\ref{sec: cubic} we bring this equation to the normal form, prove the converse, and compute the automorphism groups.
Section~\ref{sec: examples} applies this to the Markov surface, to the double Fricke surfaces, and to generalized Markov numbers.

\section{Preliminaries}\label{sec:prelim}

We recall the terminology of \cite{PeZa26} concerning completions by normal crossing divisors, adapted to our setting. It is mainly used in Sections~\ref{sec:triangle-completions}--\ref{sec:triangle-surfaces}.

A pair $(X,D)$ is an \emph{NC-pair} if $X$ is a projective surface with isolated singularities and $D$ is a nonzero reduced effective normal crossing divisor contained in the regular locus of $X$, i.e., the only singularities of $D$ are nodes.
While in \cite{PeZa26} the surface $X$ is assumed to be normal, we require only isolated singularities. Indeed, since $D$ lies in the regular locus of $X$, the intersection index of any curve on $X$ with components of $D$ is still well defined.

The \emph{dual graph} $\Gamma(D)$ of $D$ has the irreducible components of $D$ as vertices, weighted by their self-intersection indices, and the nodes of $D$ as edges; the graph is \emph{circular} if it is a cycle.
An \emph{NC-completion} of a surface $Y$ is an NC-pair $(X,D)$ together with an isomorphism $X\setminus\supp D\cong Y$; we call $D$ the \emph{boundary divisor}.
An NC-pair $(X,D)$ is \emph{minimal} if the contraction of a $(-1)$-curve on $X$ that is a component of $D$ never leads to an NC-pair.

The blowup of an NC-pair $(X,D)$ at a point of $\supp D$, endowed with the reduced total transform of $D$, is again an NC-pair; the blowup is called \emph{inner} if its center is a node of $D$ and \emph{outer} otherwise. A blowdown is inner (resp.\ outer) if the inverse blowup is.
A \emph{birational map of NC-pairs} $(X,D)\dashrightarrow(X',D')$ is a birational map $X\dashrightarrow X'$ that restricts to an isomorphism $X\setminus\supp D\to X'\setminus\supp D'$; it is \emph{inner} if it decomposes into inner blowups, inner blowdowns, and isomorphisms of NC-pairs.
Any automorphism of $Y$ extends to a unique birational map between any two NC-completions of $Y$; in particular, the group $\Bir(X,D)$ of birational selfmaps of an NC-pair is naturally identified with $\Aut(Y)$, where $Y=X\setminus\supp D$.
We denote by $\Aut(X,D)$ the group of automorphisms of $X$ preserving $D$ and by $\Aut^\natural(X,D)$ its subgroup of automorphisms preserving every component and every node of $D$. The latter is the kernel of the natural homomorphism $\Aut(X,D)\to\Aut(\Gamma(D))$, see \cite[Remark~2.3]{PeZa26}.

Given a birational map of NC-pairs $\phi\colon(X_1,D_1)\dashrightarrow(X_2,D_2)$, a \emph{decomposition} of $\phi$ is a pair of birational morphisms $\psi_i\colon(\hat X,\hat D)\to(X_i,D_i)$ of NC-pairs such that $\phi=\psi_2\circ\psi_1^{-1}$. It is \emph{relatively minimal} if no $(-1)$-curve in $\hat D$ is contracted by both $\psi_1$ and $\psi_2$.

Finally, we recall the notion of a triangulation of a weighted circular graph, see \cite[Definition~A.2]{PeZa26}. A \emph{triangulation} of a weighted circular graph $B$ is a simplicial $2$-complex $\mathcal{C}$ homeomorphic to a disc such that all vertices of $\mathcal{C}$ lie on the boundary circle, the boundary of $\mathcal{C}$ equals $B$, and every vertex of degree $k$ in the $1$-skeleton of $\mathcal{C}$ has weight $-k+1$ in $B$. If a triangulation of $B$ exists, we say that $B$ is \emph{triangulable}. The \emph{adjacency graph} of a triangulation $\mathcal{C}$ has the triangles of $\mathcal{C}$ as vertices, two of them being adjacent if they share an edge; it is always a tree. If the adjacency tree is a path, we call $\mathcal{C}$ a \emph{chain of triangles}. For example, the $3$-cycle with all weights $-1$ is triangulated by a single triangle, whereas the $2$-cycle with both weights $0$ is not triangulable.

\section{Triangle completions}\label{sec:triangle-completions}

\begin{definition}\label{def:triangle-pair}
  An NC-pair $(X,D)$ is called a \emph{triangle pair} if $D$ consists of three smooth rational $(-1)$-curves and the dual graph $\Gamma(D)$ is a cycle of length 3.
  It is called a \emph{$(0,0)$-pair} if $D$ consists of two smooth rational $0$-curves and $\Gamma(D)$ is a cycle of length 2.
  A \emph{triangle completion} (resp.\ a \emph{$(0,0)$-completion}) of a surface $Y$ is an NC-completion of $Y$ that is a triangle pair (resp.\ a $(0,0)$-pair).
\end{definition}

\begin{remark}\label{rm:triangle-vs-00}
  Let $(X,D)$ be a triangle pair with $D=L_1+L_2+L_3$.
  The contraction of $L_1$ is an inner blowdown onto a $(0,0)$-pair: the images of $L_2$ and $L_3$ are $0$-curves meeting transversally at two points, namely the image of $L_1$ and the node $L_2\cap L_3$.
  Conversely, the inner blowup of a $(0,0)$-pair at any of its two nodes is a triangle pair.
  In particular, a surface admits a triangle completion if and only if it admits a $(0,0)$-completion.
  Note also that a $(0,0)$-pair is minimal, since its components are not $(-1)$-curves, whereas a triangle pair is not.
\end{remark}

We will describe birational maps between triangle and $(0,0)$-completions of a fixed surface $Y$ in terms of $\PGL(2,\ZZ)$.
In the following proposition we recall some additional results from \cite{PeZa26}.

\begin{proposition}\label{pr:00-completions-PZ}
  Let $(X_1,D_1)$ be a $(0,0)$-completion of a surface $Y$ and let $\phi\colon (X_1,D_1)\dashrightarrow (X_2,D_2)$ be a birational map of NC-pairs such that $(X_2,D_2)$ is minimal.
  Then the following hold.
  \begin{enumerate}[(i)]
    \item The NC-pair $(X_2,D_2)$ is a $(0,0)$-completion of $Y$.\label{pr:00-completions-PZ:00}
    \item There exists a decomposition $\phi=\psi_2\circ\psi_1^{-1}$, where each morphism
  $\psi_i\colon (\hat{X},\hat D)\to (X_i,D_i)$, $i=1,2$, is a composition of inner blowups. \label{pr:00-completions-PZ:dominating-pair}
    \item For any such decomposition the dual graph $\Gamma(\hat D)$ is circular.\label{pr:00-completions-PZ:circular}
    \item If $\phi$ is not an isomorphism, then $\Gamma(\hat D)$ is triangulable. 
    Moreover, if $(\hat{X},\hat D)$ is relatively minimal, then $\Gamma(\hat D)$ admits a triangulation which is a chain of triangles.\label{pr:00-completions-PZ:chain}
  \end{enumerate}
\end{proposition}
\begin{proof}
  Item \ref{pr:00-completions-PZ:00} follows from \cite[Proposition~A.4]{PeZa26}.
  Items \ref{pr:00-completions-PZ:dominating-pair} and \ref{pr:00-completions-PZ:circular} follow from \cite[Proposition~4.15, Theorem~4.17]{PeZa26}.
  Item \ref{pr:00-completions-PZ:chain} follows from Case~6 in the proof of \cite[Proposition~A.4]{PeZa26}.
\end{proof}

\begin{remark}\label{rm:canonical-triangulation}
  Consider an inner birational morphism $\psi\colon(\hat X,\hat D)\to(X,D)$ and assume that $\Gamma(D)$ is triangulable.
  Then there is a triangulation of $\Gamma(\hat D)$ obtained recursively by adding a triangle for each inner blowup in a decomposition of $\psi$ into inner blowups, starting with a triangulation of $\Gamma(D)$. 
  The added triangle is adjoined to the edge corresponding to the blown up node of the boundary divisor.
\end{remark}

\begin{lemma}\label{lm:triangulation-unique}
  A weighted circular graph admits at most one triangulation.
\end{lemma}
\begin{proof}
  Let $\mathcal{C}$ be a triangulation of a weighted circular graph $B$; we proceed by induction on the number of vertices.
  If $B$ has three vertices, then $\mathcal{C}$ is the single triangle spanned by them.
  Otherwise, choose a vertex $v$ of weight $-1$; it exists, since the adjacency tree of $\mathcal{C}$ has a leaf, and a leaf triangle contains a vertex belonging to no other triangle, whose degree is thus $2$.
  Conversely, any vertex $v$ of weight $-1$ has degree $2$, so in every triangulation of $B$ it spans a triangle $\Delta_v$ with its two neighbours $u,w$ in $B$, and $\Delta_v$ is a leaf of the adjacency tree.
  Removing $\Delta_v$ yields a triangulation of the circular graph obtained from $B$ by removing $v$, joining $u$ to $w$, and increasing their weights by $1$.
  The latter triangulation is unique by the induction hypothesis, hence so is $\mathcal{C}$.
\end{proof}

Thus, if the dual graph $\Gamma(D)$ of an NC-pair $(X,D)$ is triangulable, then its triangulation is unique, and we denote it by $T(D)$.

\begin{definition}
 Let $p$ be a node of the boundary divisor $D$ of an NC-completion $(X,D)$. 
 The exceptional curves of all iterative blowups at $p$ and at infinitesimally near nodes of the boundary divisor are called \emph{inner components at $p$}; 
 the components of $D$ together with the inner components at all nodes of $D$ are the \emph{inner components of $D$}.
 
 Let $Y$ be an open surface admitting a triangle completion.
 Then by \emph{inner components} of $Y$ we mean the inner components of $D$ for an arbitrary triangle completion $(X,D)$ of $Y$.
 This is well defined due to Remark~\ref{rm:inner-identification} below.
\end{definition}

\begin{remark}\label{rm:inner-identification}
  Let $\phi\colon(X_1,D_1)\dashrightarrow(X_2,D_2)$ be an inner birational map of NC-pairs.
  Then there is a natural bijection between the inner components of $D_1$ and the inner components of $D_2$.
  Indeed, let $\psi_i\colon(\hat X,\hat D)\to(X_i,D_i)$ be inner morphisms for $i=1,2$ such that $\psi=\psi_2\circ\psi_1^{-1}$ is a decomposition of $\psi$.
  Then the inner components of $D_i$ are exactly the components of $\hat D$ up to taking strict transforms, $i=1,2$.
  The uniqueness of the bijection can also be seen from the fact that inner components of $D_i$ correspond to (distinct) divisorial valuations of $Y:=X_1\setminus\supp D_1$.
\end{remark}

\begin{lemma}\label{lm:completion-unique}
  A $(0,0)$- or triangle completion of a surface $Y$ is uniquely determined, up to a unique isomorphism over $Y$, by the set of inner components of $Y$ contained in the boundary divisor.
\end{lemma}
\begin{proof}
  Let $\phi\colon(X_1,D_1)\dashrightarrow(X_2,D_2)$ be an inner birational map of completions of $Y$. Assume that $D_1$ and $D_2$ define the same set of inner components of $Y$.
  Take a relatively minimal decomposition $\phi=\psi_2\circ\psi_1^{-1}$, $\psi_i\colon(\hat X,\hat D)\to(X_i,D_i)$. 
  Then the $(-1)$-curve of $\Gamma(\hat D)$ first contracted by $\psi_1$ is also contracted by $\psi_2$, a contradiction with the relative minimality.
\end{proof}

Now we are ready to introduce the main combinatorial object of this section.

\begin{definition}\label{def:complex-T}
  Let $Y$ be a surface admitting a $(0,0)$-completion.
  The \emph{triangle complex} $T(Y)$ is the two-dimensional abstract simplicial complex defined as follows.
  The vertices of $T(Y)$ are the inner components of $Y$.
  A pair (resp.\ triple) of vertices is an \emph{edge} (resp.\ a \emph{triangle}) of $T(Y)$ if it is the set of inner components of $Y$ defined by the boundary components of a $(0,0)$-completion (resp.\ triangle completion) of $Y$.

  Note that, by Lemma~\ref{lm:completion-unique}, the edges (resp.\ triangles) of $T(Y)$ are in natural bijection with the $(0,0)$-completions (resp.\ triangle completions) of $Y$ up to isomorphism over $Y$.
\end{definition}

\begin{lemma}\label{lm:complex-T-basic}
  The following assertions hold.
  \begin{enumerate}[(i)]
  \item $T(Y)$ is a well-defined simplicial complex, i.e., every pair of vertices of a triangle of $T(Y)$ is an edge of $T(Y)$.\label{lm:complex-T-basic:simplicial}
  \item Every edge of $T(Y)$ is contained in exactly two triangles of $T(Y)$.\label{lm:complex-T-basic:two-triangles}
  \end{enumerate}
\end{lemma}
\begin{proof}
  \ref{lm:complex-T-basic:simplicial} The contraction of a boundary component of a triangle completion of $Y$ is a $(0,0)$-completion of $Y$, see Remark~\ref{rm:triangle-vs-00}, whose boundary defines the corresponding pair of vertices.

  \ref{lm:complex-T-basic:two-triangles} Let an edge of $T(Y)$ be defined by a $(0,0)$-completion $(X,D)$ of $Y$.
  The inner blowups of $(X,D)$ at its two nodes are triangle completions of $Y$ whose triangles contain the given edge, and their exceptional curves are distinct inner components of $Y$; this gives two distinct triangles.
  Conversely, let a triangle completion $(X',D')$ of $Y$ define a triangle containing the given edge.
  Then the contraction of the third boundary component of $D'$ is a $(0,0)$-completion defining the same edge, and by Lemma~\ref{lm:completion-unique} $(X',D')$ is isomorphic over $Y$ to the inner blowup of $(X,D)$ at one of its two nodes.
\end{proof}

The following lemma relates the finite triangulations of Remark~\ref{rm:canonical-triangulation} to the infinite complex $T(Y)$.

\begin{lemma}\label{lm:complex-T-realization}
  Let $(X,D)$ be a $(0,0)$-completion of $Y$.
  \begin{enumerate}[(i)]
  \item For every inner birational morphism $\psi\colon(\hat X,\hat D)\to(X,D)$, the triangulation $T(\hat D)$ is a finite subcomplex of $T(Y)$ under the identification of the components of $\hat D$ with inner components of $Y$.\label{lm:complex-T-realization:sub}
  \item Let $\phi\colon(X_1,D_1)\dashrightarrow(X_2,D_2)$ be a birational map of $(0,0)$-completions of $Y$ that is not an isomorphism and $\psi_i\colon(\hat X,\hat D)\to(X_i,D_i)$, $i=1,2$ be a relatively minimal decomposition $\phi=\psi_2\circ\psi_1^{-1}$.
  Then $T(\hat D)$ is a chain of triangles connecting the edges $e_1$ and $e_2$ of $T(Y)$ corresponding to $(X_1,D_1)$ and $(X_2,D_2)$ respectively, i.e., these edges belong to the end triangles of the chain.\label{lm:complex-T-realization:chain}
  \end{enumerate}
\end{lemma}
\begin{proof}
  \ref{lm:complex-T-realization:sub}
  The first blowdown in a decomposition of $\psi$ contracts a $(-1)$-component of $\hat D$, and this removes the leaf triangle of $T(\hat D)$ at the corresponding vertex. 
  The claim then follows by induction on the number of blowups in $\psi$, the base case being a single inner blowup of $(X,D)$, i.e., a triangle completion of $Y$.

  \ref{lm:complex-T-realization:chain}
  By the relative minimality, the only $(-1)$-vertices of $\Gamma(\hat D)$ are strict transforms of components of $D_1$ and $D_2$. Since the leaf triangles of $T(\hat D)$ are exactly those spanned by a $(-1)$-vertex and its two neighbours, the adjacency tree of $T(\hat D)$ is a path whose end triangles contain the edges $e_1$ and $e_2$.
\end{proof}

\begin{lemma}\label{lm:complex-T-tree}
  The adjacency graph of triangles of $T(Y)$, in which two triangles are adjacent if they share an edge, is an infinite $3$-regular tree.
\end{lemma}
\begin{proof}
  It is clear from Lemmas~\ref{lm:complex-T-basic} and Lemma~\ref{lm:complex-T-realization} that the adjacency graph, say $A$, is infinite and $3$-regular. 
  Since any two $(0,0)$-completions are connected by an inner birational map, the graph $A$ is connected as well.

  Finally, passing to an adjacent triangle replaces the vertex opposite to the shared edge by the exceptional component of an inner blowup, see the proof of Lemma~\ref{lm:complex-T-basic}\ref{lm:complex-T-basic:two-triangles}. Since the inner components of $Y$ are pairwise distinct, a non-backtracking path in the adjacency graph never returns to its initial triangle, and the graph is a tree.
\end{proof}

\begin{notation}\label{nt:markings}
  \begin{enumerate}[(i)]
  \item Denote by $\PP(\ZZ^2)$ the set of pairs $\{v,-v\}$ of opposite primitive vectors in $\ZZ^2$.
  We will write a primitive vector of $\ZZ^2$, meaning the corresponding element of $\PP(\ZZ^2)$.\label{nt:markings:PP}
  \item Let $(X,D)$ be an NC-pair and $p$ be a node of $D$ that is the intersection of two distinct components $D_1$ and $D_2$ marked with vectors $v_1,v_2\in\ZZ^2$ that comprise a basis.
  Then we mark the inner components at $p$ iteratively: the exceptional component of the blowup at $p$ is marked by $v_1+v_2$, and we apply the same procedure at the two nodes of the exceptional component.
  In particular, the pair of markings at every node of every iterated inner blowup over $p$ is again a basis of $\ZZ^2$.\label{nt:markings:node}
  \item Let now $(X,D)$ be a $(0,0)$-completion of $Y$, where $D$ consists of components $D_1,D_2$ intersecting at the nodes $p$ and $p'$.
  Firstly, we mark all inner components at $p$ starting with $(1,0)$ for $D_1$ and $(0,1)$ for $D_2$.
  Secondly, we mark all inner components at $p'$ starting with $(1,0)$ for $D_1$ and $(0,-1)$ for $D_2$.
  The two markings of $D_2$ agree up to multiplication by $-1$, so we consider all markings as elements of $\PP(\ZZ^2)$.
  Thus, we will denote vertices of $T(Y)$ and the corresponding elements of $\PP(\ZZ^2)$ by the same letters.\label{nt:markings:00}
  \item Since a pair $u,v\in\ZZ^2$ is a basis if and only if all four pairs $\pm u,\pm v$ are, it is well defined whether a pair of elements of $\PP(\ZZ^2)$ is a basis of $\ZZ^2$.\label{nt:markings:basis}
  \end{enumerate}
\end{notation}

\begin{lemma}\label{lm:marking-bijection}
  The marking of Notation~\ref{nt:markings}\ref{nt:markings:00} is a bijection between the inner components of $D$ and $\PP(\ZZ^2)$.
\end{lemma}
\begin{proof}
  The components $D_1$ and $D_2$ are marked by $(1,0)$ and $(0,1)$.
  We claim that the inner components at $p$ are marked bijectively by the classes $\pm(a,b)$ with $a,b>0$.
  Then, symmetrically, the inner components at $p'$ are marked bijectively by the classes $\pm(a,b)$ with $a>0>b$, and the assertion follows, since any primitive vector with a zero coordinate is equal to $\pm(1,0)$ or $\pm(0,1)$.

  The markings at $p$ of any inner component and of the two components meeting it at a node form a basis of $\ZZ^2$ lying in the closed positive quadrant.
  We show by induction on $a+b$ that any primitive vector $(a,b)$ with $a,b>0$ is the marking of exactly one inner component at $p$.
  The base is $(1,1)$, which marks exactly the exceptional component of the blowup at $p$.
  For $a+b\geq3$ we have $a\neq b$, and $(a,b)=u+v$ for a basis $u,v$ of $\ZZ^2$ in the closed positive quadrant if and only if $\{u,v\}=\{(a,b)-w,w\}$, where $w$ is the unique vector satisfying $0\leq w\leq (a,b)$ coordinatewise and $\det(w,(a,b))=1$.
  Such $w$ exists and is unique: since $\gcd(a,b)=1$, the solutions $w=(x,y)$ of $xb-ya=1$ form a coset of $\ZZ(a,b)$ in $\ZZ^2$, so exactly one of them satisfies $1\leq x\leq a$, and for it $y=\frac{xb-1}{a}\in\left[\frac{b-1}{a},\,b-\frac1a\right]\subseteq[0,b)$.
  Hence there is exactly one node whose blowup produces the marking $(a,b)$, and the components meeting at this node are marked by $w$ and $(a,b)-w$ with smaller coordinate sums.
  By the induction hypothesis this node exists and is unique.
\end{proof}

We depict the simplicial complex $T(Y)$ marked by $\PP(\ZZ^2)$ and the adjacency tree of triangles in Fig.~\ref{fig:tree}.

\begin{figure}
\begin{tikzpicture}[scale=1.2,  every node/.style={font=\footnotesize}]
\def\L{2}
\pgfmathsetmacro{\h}{\L * sqrt(3)/2} 
\tikzset{vertex/.style={rectangle, fill=white, inner sep=1pt, fill opacity=0.8, text opacity=1}} 

\newcommand{\setcoordinates}[2]{%
    \pgfmathtruncatemacro{\up}{mod(#1 + #2, 2)}
    \ifnum\up=0
        \pgfmathsetmacro{\x}{#2 * \h}
        \pgfmathsetmacro{\y}{#1 * \L / 2}
        \coordinate (A#1#2) at (\x,\y);
        \coordinate (C#1#2) at ($(A#1#2)+(\h,\L/2)$);
    \else
        \pgfmathsetmacro{\x}{#2 * \h + \h}
        \pgfmathsetmacro{\y}{#1 * \L / 2}
        \coordinate (A#1#2) at (\x,\y);
        \coordinate (C#1#2) at ($(A#1#2)+(-\h,\L/2)$);
    \fi
    \coordinate (B#1#2) at ($(A#1#2)+(0,\L)$);
    \coordinate (T#1#2) at ($ (B#1#2)!0.333!(A#1#2) + (B#1#2)!0.333!(C#1#2) -(B#1#2) $);
}
\newcommand{\drawtriangle}[2]{%
    \setcoordinates{#1}{#2}
    \draw[gray, thin, -latex] (A#1#2) -- (B#1#2) -- (C#1#2) -- cycle;
}

\newcommand{\drawtreeedge}[4]{%
    \setcoordinates{#1}{#2}
    \setcoordinates{#3}{#4}
    \draw[red] (T#1#2) -- (T#3#4);
}

\newcommand{\drawtreeedgedashed}[4]{%
    \setcoordinates{#1}{#2}
    \setcoordinates{#3}{#4}
    \coordinate (end) at ($(T#1#2)!0.666!(T#3#4)$);
    \draw[red,dashed] (T#1#2) -- (end);
}

\drawtriangle{0}{0}
\drawtriangle{1}{0}
\drawtriangle{2}{0}
\drawtriangle{0}{-1}
\drawtriangle{2}{-1}
 
\drawtriangle{1}{1}
 \drawtriangle{2}{1}
 \drawtriangle{3}{1}
 \drawtriangle{0}{1}
 \drawtriangle{-1}{1}
 \drawtriangle{2}{2}
 \drawtriangle{0}{2}
 \draw[black, thick, -latex] (A11) -- (B11) -- cycle;

\node[vertex] at (A11) {$(0,1)$};
\node[vertex] at (B11) {$(1,0)$};
\node[vertex] at (C11) {$(1,1)$};
\node[vertex] at (C10) {$(1,-1)$};
\node[vertex] at (A00) {$(1,-2)$};
\node[vertex] at (B20) {$(2,-1)$};
\node[vertex] at (C0-1) {$(2,-3)$};
\node[vertex] at (C2-1) {$(3,-2)$};

\node[vertex] at (A-11) {$(1,3)$};
\node[vertex] at (A01) {$(1,2)$};
\node[vertex] at (B21) {$(2,1)$};
\node[vertex] at (B31) {$(3,1)$};
\node[vertex] at (C22) {$(3,2)$};
\node[vertex] at (C02) {$(2,3)$};

\drawtreeedge{0}{0}{1}{0}
 \drawtreeedge{1}{0}{2}{0}
 \drawtreeedge{1}{0}{1}{1}

 \drawtreeedge{0}{0}{0}{-1}
 \drawtreeedge{2}{-1}{2}{0}

 \drawtreeedge{2}{1}{1}{1}
 \drawtreeedge{1}{1}{0}{1}

\drawtreeedge{2}{1}{3}{1}
 \drawtreeedge{2}{1}{2}{2}

  \drawtreeedge{-1}{1}{0}{1}
 \drawtreeedge{0}{2}{0}{1}

\drawtreeedgedashed{0}{-1}{-1}{-1}
\drawtreeedgedashed{0}{-1}{1}{-1}
\drawtreeedgedashed{2}{-1}{1}{-1}
\drawtreeedgedashed{2}{-1}{3}{-1}

\drawtreeedgedashed{0}{0}{-1}{0}
\drawtreeedgedashed{2}{0}{3}{0}

\drawtreeedgedashed{-1}{1}{-2}{1}
\drawtreeedgedashed{-1}{1}{-1}{0}
\drawtreeedgedashed{3}{1}{4}{1}
\drawtreeedgedashed{3}{1}{3}{0}

\drawtreeedgedashed{2}{2}{3}{2}
\drawtreeedgedashed{2}{2}{1}{2}
\drawtreeedgedashed{0}{2}{-1}{2}
\drawtreeedgedashed{0}{2}{1}{2}

\end{tikzpicture}
\caption{The simplicial complex $T(Y)$ depicted in black.
Each vertex is an inner component of $Y$ marked by an element of $\PP(\ZZ^2)$.
The adjacency tree of triangles is depicted in red. 
The starting $(0,0)$-completion is emphasized by a thick edge.\label{fig:tree}}  
\end{figure}
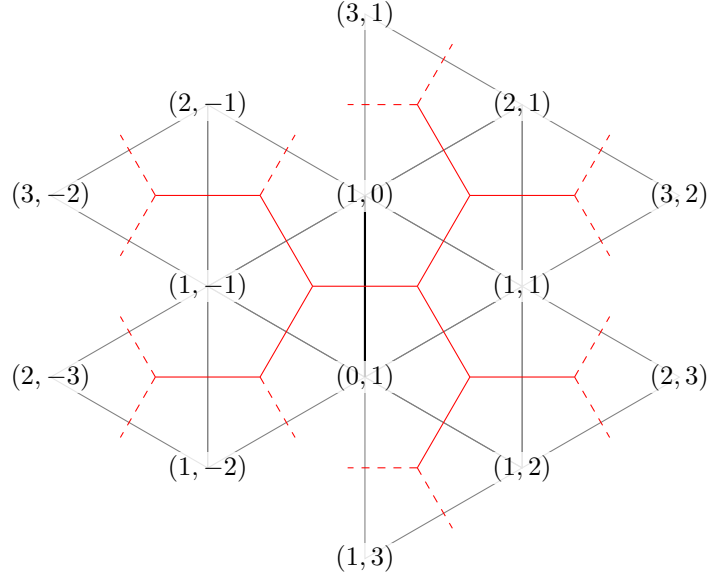

\begin{figure}
\begin{tikzpicture}[scale=1.5]
    \draw[->] (-0.5, 0) -- (5, 0) node[right] {$x$};
    \draw[->] (0, -0.5) -- (0, 5) node[above] {$y$};

    \foreach \x in {0,1,...,4} {
        \foreach \y in {0,1,...,4} {
            \fill (\x,\y) circle (2pt);
            \node[above right, font=\tiny] at (\x,\y) {$(\x,\y)$};
        }
    }

    \def\e{5}
    \foreach \a in {0,...,\e} {
        \foreach \b in {0,...,\e} {
            \foreach \c in {0,...,\e} {
                \foreach \d in {0,...,\e} {
                    \pgfmathparse{(\a == \c && \b == \d) ? 1 : 0}
                    \ifnum\pgfmathresult=0
                        \pgfmathparse{int(\a*\d - \b*\c)}
                        \ifnum\pgfmathresult=1
                            \draw[gray, thick] (\a,\b) -- (\c,\d);
                        \fi
                    \fi
                }
            }
        }
    }
\end{tikzpicture}
\caption{The positive quadrant of $\ZZ^2$ tiled by triangles of the right half of $T(Y)$.\label{fig:quadrant}}  
\end{figure}
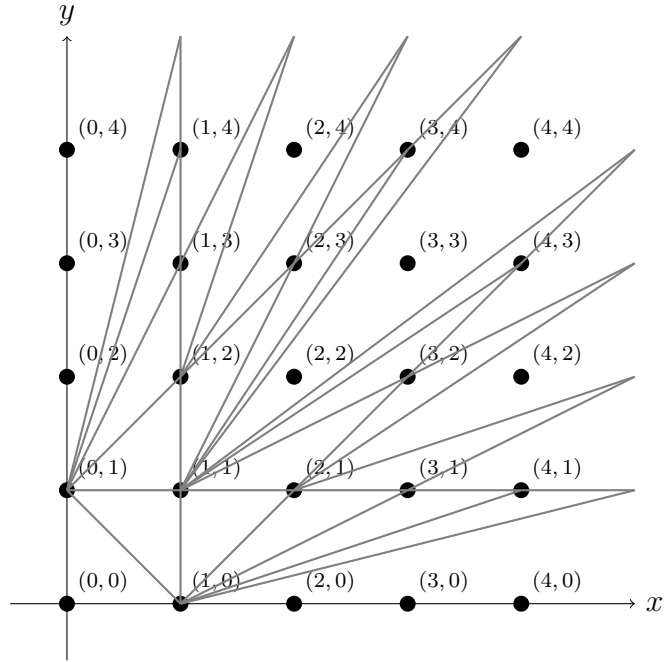

\begin{lemma}\label{lm:markings}
  Let $T(Y)$ be marked. Then the following assertions hold.
  \begin{enumerate}[(i)]
  \item Given an edge $(u,v)$ in $T(Y)$, the opposite vertices of adjacent triangles are $u+v,u-v$.
  \item The markings of $T(Y)$ are uniquely determined by markings of any triangle.
  \item Let us take an edge $(u,v)$ in $T(Y)$ and construct a new marking of $T(Y)$ by starting with $(1,0)$ and $(0,1)$ for $u$ and $v$ respectively.
  Then there exists a unique $g\in\PGL(2,\ZZ)$ such that for any vertex with the new marking $m$ the original marking equals $gm$.
  In particular, $g$ is the class of a matrix whose columns represent $u$ and $v$.
 \end{enumerate}
\end{lemma}
\begin{proof}
  (i) Every triangle of $T(Y)$ has vertices marked by $m_1$, $m_2$, and $m_1+m_2$ for suitable representatives $m_1,m_2\in\ZZ^2$ comprising a basis.
  Thus each of the two vertices opposite to the edge $(u,v)$ equals $\pm u\pm v$, i.e., $u+v$ or $u-v$ in $\PP(\ZZ^2)$; these two classes are distinct, whence the assertion.

  (ii) By (i), the markings of a triangle determine the markings of the vertices opposite to its edges, i.e., of the adjacent triangles.
  Since the adjacency tree of triangles is connected and every vertex of $T(Y)$ belongs to a triangle, the markings of a single triangle determine the markings of all vertices.

  (iii) The new marking $\nu$ is Notation~\ref{nt:markings}\ref{nt:markings:00} applied to the $(0,0)$-completion corresponding to $(u,v)$, so assertion (i) and Lemma~\ref{lm:marking-bijection} hold for $\nu$ as well; denote the original marking by $\mu$.
  Let $o$ be the vertex with $\nu(o)=(1,1)$; by (i) we may choose representatives $m_u,m_v\in\ZZ^2$ of $\mu(u),\mu(v)$ with $\mu(o)=m_u+m_v$, and we let $g\in\PGL(2,\ZZ)$ be the class of the matrix with columns $m_u$ and $m_v$.
  Since $g$ is linear, the labeling $g\nu$ satisfies (i), and it agrees with $\mu$ on the triangle $(u,v,o)$; hence $\mu=g\nu$ by the reconstruction of (ii).
  Finally, if $g'\nu=g\nu$, then $g^{-1}g'$ fixes every element of $\PP(\ZZ^2)$, since $\nu$ is surjective; hence it is the identity of $\PGL(2,\ZZ)$ and $g=g'$.
\end{proof}
 
\begin{remark}\label{rm:elementary-moves}
  Let $(u,v)$ be an edge of the marked complex $T(Y)$.
  By Lemma~\ref{lm:markings}(i), the vertices opposite to it in the two adjacent triangles are $u+v$ and $u-v$.
  \begin{enumerate}[(i)]
  \item The \emph{elementary transformations} of $\ZZ^2$ mapping $v\mapsto v\pm u$ and fixing $u$, resp.\ mapping $u\mapsto u\pm v$ and fixing $v$, are given in the basis $(u,v)$ by the matrices
  \[
  \begin{pmatrix}1&0\\\pm1&1\end{pmatrix},
  \qquad
  \begin{pmatrix}1&\pm1\\0&1\end{pmatrix}
  \in\PSL(2,\ZZ).
  \]
  As automorphisms of $T(Y)$ they map the triangle $(u,v,u+v)$ to one of four adjacent ones, keeping the orientation of $T(Y)$.\label{rm:elementary-moves:elementary}
  \item There are two possibilities for the new marking in Lemma~\ref{lm:markings}(iii) depending on the choice of nodes of the boundary divisor of the $(0,0)$-completion by components $u,v$; the corresponding elements $g,g'\in \PGL(2,\ZZ)$ differ by right multiplication by $\left(\begin{smallmatrix}1&0\\0&-1\end{smallmatrix}\right)$, and exactly one of them belongs to $\PSL(2,\ZZ)$.
  The two markings differ by the mirror symmetry of $T(Y)$ relative to the edge $(u,v)$, which swaps the opposite vertices $u+v$ and $u-v$.\label{rm:elementary-moves:symmetry}
  \end{enumerate}
\end{remark}

\begin{lemma}\label{lm:basis-quadrant}
    If $u,v$ is a basis of $\ZZ^2$, then either $\{u,v\}$ or $\{u,-v\}$ belong to the same closed quadrant of $\ZZ^2$.
\end{lemma}
\begin{proof}
  If some coordinate of $u$ or $v$ vanishes, the assertion is checked directly.
  Otherwise, assume the contrary; then, up to swapping the vectors and multiplying them by $-1$, we have $u=(a,b), v=(-c,d)$ for $a,b,c,d>0$.
  In this case $\det\left(\begin{smallmatrix}a&b\\-c&d\end{smallmatrix}\right)=ad+bc\ge2$, and the pair $u,v$ is not a basis of $\ZZ^2$.
\end{proof}

\begin{corollary}\label{cr:basis-edge}
   A pair of primitive vectors in $\ZZ^2$ is a basis if and only if the corresponding vertices of $T(Y)$ are adjacent.
\end{corollary}
\begin{proof}
    If two vertices are adjacent, then they belong to a triangle of $T(Y)$, whose vertices are marked by $m_1$, $m_2$, $m_1+m_2$ for suitable representatives $m_1,m_2\in\ZZ^2$ comprising a basis, see the proof of Lemma~\ref{lm:markings}(i); any two of these three vectors comprise a basis of $\ZZ^2$.

    Let $u,v$ be non-adjacent vertices.
    Since no triangle of $T(Y)$ contains both, a shortest path $\Delta_0,\ldots,\Delta_n$ in the adjacency tree of triangles with $u\in\Delta_0$ and $v\in\Delta_n$ has length $n\geq1$.
    Denote by $e$ the common edge of $\Delta_0$ and $\Delta_1$; by minimality $u,v\notin e$, so $u$ is the vertex of $\Delta_0$ opposite to $e$.
    By Lemma~\ref{lm:markings}(iii) and Remark~\ref{rm:elementary-moves}\ref{rm:elementary-moves:symmetry} we may change the marking so that $e=((1,0),(0,1))$ and $u=(1,-1)$.
    Now, $v$ or $-v$ lies in the positive quadrant of $\ZZ^2$, since otherwise the path $\Delta_0,\ldots,\Delta_n$ would cross the edge $e$ again, which is impossible by the minimality of the path.
     Hence neither $\{u,v\}$ nor $\{u,-v\}$ lies in a common closed quadrant, and $u,v$ is not a basis by Lemma~\ref{lm:basis-quadrant}.
\end{proof}

\begin{corollary}\label{cr:aut-T-PGL}
  The identification of vertices of $T(Y)$ with $\PP(\ZZ^2)$ induces the isomorphism
  of the natural action of $\Aut(T(Y))$ on $T(Y)$ 
  with the natural action of $\PGL(2,\ZZ)$ on $\PP(\ZZ^2)$.
\end{corollary}
\begin{proof}
  Denote by $\mu$ the marking, which is a bijection between the vertices of $T(Y)$ and $\PP(\ZZ^2)$ by Lemma~\ref{lm:marking-bijection}.
  Since $g\in\PGL(2,\ZZ)$ preserves bases of $\ZZ^2$ and triples of classes $m_1$, $m_2$, $m_1+m_2$, the bijection $\mu^{-1}g\mu$ preserves edges and triangles of $T(Y)$ by Corollary~\ref{cr:basis-edge} and Lemma~\ref{lm:markings}(i).
  This action of $\PGL(2,\ZZ)$ is faithful: an element fixing $(1,0)$, $(0,1)$, $(1,1)$ in $\PP(\ZZ^2)$ is represented by $\pm E$.

  Conversely, for $\phi\in\Aut(T(Y))$ the labeling $\mu\phi$ satisfies Lemma~\ref{lm:markings}(i), so, as in the proof of Lemma~\ref{lm:markings}(iii), there is $g\in\PGL(2,\ZZ)$ such that $\mu\phi$ and $g\mu$ agree on the triangle $((1,0),(0,1),(1,1))$; by the reconstruction of Lemma~\ref{lm:markings}(ii) they coincide, i.e., $\phi=\mu^{-1}g\mu$.
\end{proof}

\begin{remark}\label{rm:farey}
  By Corollary~\ref{cr:basis-edge}, the marking identifies $T(Y)$ with the Farey tessellation of the hyperbolic plane: vertices are the classes $\pm(p,q)\in\PP(\ZZ^2)$, i.e., the elements $p/q$ of $\QQ\cup\{\infty\}$, and edges are the pairs comprising a basis of $\ZZ^2$, cf.~\cite[Sec.~3]{FKS12}.
  Corollary~\ref{cr:aut-T-PGL} thus recovers the well-known fact that the automorphism group of the Farey tessellation is $\PGL(2,\ZZ)$.
\end{remark}

\begin{theorem}\label{th:aut-action}
  Let $(X,D)$ be a $(0,0)$-completion or a triangle completion of $Y$.
  Then the action of $\Aut(Y)\cong\Bir(X,D)$ on the inner components of $Y$ induces an action on $T(Y)$ by automorphisms with kernel $\Aut^\natural(X,D)$.
  In particular, we have the exact sequence
  \[
  1\to \Aut^\natural(X,D) \to \Aut(Y) \to \Aut(T(Y))\cong\PGL(2,\ZZ).
  \]
\end{theorem}
\begin{proof}
  An automorphism $g\in\Aut(Y)$ permutes the inner components of $Y$, regarded as divisorial valuations of $Y$, see Remark~\ref{rm:inner-identification}; composing $(0,0)$- and triangle completions of $Y$ with $g$, we see that this permutation preserves the sets of edges and triangles of $T(Y)$, and Corollary~\ref{cr:aut-T-PGL} yields the homomorphism $\mu\colon\Aut(Y)\to\Aut(T(Y))\cong\PGL(2,\ZZ)$.
  If $g\in\Aut^\natural(X,D)$, then, lifting $g$ to iterated inner blowups, we see by induction that $g$ preserves every inner component of $Y$, so $\mu(g)=\id$.
  Conversely, assume that $\mu(g)=\id$, i.e., $g$ fixes every inner component of $Y$.
  In particular, $g$ preserves the set of inner components defined by the components of $D$, hence extends to an automorphism of $(X,D)$ by Lemma~\ref{lm:completion-unique}.
  This automorphism preserves every component of $D$.
  It also preserves every node of $D$, since the exceptional component of the blowup at a node is a $g$-fixed inner component of $Y$ and differs from node to node.
  Hence $g\in\Aut^\natural(X,D)$.
\end{proof}

\section{Fibrations}\label{sec:fibrations}

In this section we attach to every vertex of the triangle complex a $\PP^1$-fibration.
Throughout the section, $Y$ is a normal surface admitting a $(0,0)$-completion and $T(Y)$ is its triangle complex, see Definition~\ref{def:complex-T}.
Note that every NC-completion $(X,D)$ of $Y$ is then normal, since $D$ lies in the regular locus of $X$.
We say that morphisms $\pi_1\colon Z\to B_1$ and $\pi_2\colon Z\to B_2$ of a surface $Z$ onto curves define the \emph{same fibration} if $\pi_2=\beta\circ\pi_1$ for an isomorphism $\beta\colon B_1\to B_2$; equivalently, if they have the same fibers.
By a \emph{$\PP^1$-fibration} of a projective surface $X$ we mean a morphism $X\to\PP^1$ whose general fiber is isomorphic to $\PP^1$.

Let us first describe the $\PP^1$-fibrations attached to a triangle completion.

\begin{lemma}\label{lm:fibrations}
Let $(X,D)$ be a triangle pair with $X$ normal and $D=L_1+L_2+L_3$. Then we have the $\PP^1$-fibrations
$$\pi_1,\pi_2,\pi_3\colon X\to \PP^1,$$
which are defined by the linear systems $|L_2+L_3|, |L_1+L_3|, |L_1+L_2|$ respectively.
In particular, a general fiber $F$ of $\pi_i$ meets $D$ only along $L_i$, and $F\cdot L_i=2$, $i=1,2,3$.
\end{lemma}
\begin{proof}
By Remark~\ref{rm:triangle-vs-00}, the contraction of $L_1$ is an inner blowdown $\sigma\colon(X,D)\to(X',D')$ onto a $(0,0)$-pair.
Denote by $C$ and $C'$ the images of $L_2$ and $L_3$, so that $\sigma^*C=L_1+L_2$ and $C\cdot C'=2$.

A resolution of singularities of $X'$ is an isomorphism near $D'$, which lies in the regular locus.
By \cite[Chapter~V, Proposition~4.3(i)]{BPV04}, the 0-curve $C$ on the resolution is a fiber of a $\PP^1$-fibration over a smooth curve $B$ (possibly with reducible fibers).
The exceptional curves of the resolution do not meet the fiber $C$, hence lie in fibers, and the fibration descends to a $\PP^1$-fibration $\pi_C\colon X'\to B$, the surface $X'$ being normal.
The rational curve $C'$ is not contained in a fiber, since $C'\cdot C=2$. Thus, the image of $C'$ equals $B$, and $B\cong\PP^1$.

Now all fibers of $\pi_C$ are linearly equivalent; conversely, every element of the complete linear system $|C|$ is a fiber.
Hence $|C|=\pi_C^*|\cO_{\PP^1}(1)|$ is one-dimensional and defines $\pi_C$.
Its preimage in $X$ is $|L_1+L_2|$, which thus defines the $\PP^1$-fibration $\pi_3:=\pi_C\circ\sigma$ of $X$.
Similarly, we obtain the fibrations $\pi_1$ and $\pi_2$.
The assertions on a general fiber $F$ of $\pi_3$ follow from the equalities $F\cdot L_3=(L_1+L_2)\cdot L_3=2$ and $F\cdot L_1=F\cdot L_2=0$.
\end{proof}

\begin{example}\label{ex:markov-fibrations}
  Consider the Markov surface $M=\{xyz=x^2+y^2+z^2\}\subset \AA^3$ and its closure $\overline{M}\subset \PP^3$.
  Then the fibrations $\pi_1$, $\pi_2$, and $\pi_3$ of $\overline{M}$ are the extensions of the coordinate functions $x,y,z\colon M\to\AA^1\subset\PP^1$.
\end{example}

By Lemma~\ref{lm:fibrations}, each component $L_i$ of the boundary divisor of a triangle completion carries a $\PP^1$-fibration $\pi_i$, whose general fibers meet the boundary only along $L_i$.
We are going to show that the restriction of $\pi_i$ to $Y$ depends only on the vertex of $T(Y)$ defined by $L_i$ and not on the triangle completion.
We first describe the star of a vertex of $T(Y)$ in the following lemma that complements Lemma~\ref{lm:complex-T-tree}.

\begin{lemma}\label{lm:vertex-fan}
  Let $v$ be a vertex of $T(Y)$. Then the vertices adjacent to $v$ can be numbered as $u_a$, $a\in\ZZ$, so that the triangles of $T(Y)$ containing $v$ are exactly $\Delta_a:=(v,u_a,u_{a+1})$, $a\in\ZZ$.
  In particular, the triangles $\Delta_{a-1}$ and $\Delta_a$ are adjacent along the edge $(v,u_a)$, so any two triangles containing $v$ are connected by a sequence of triangles containing $v$ in which consecutive ones share an edge containing $v$.
\end{lemma}
\begin{proof}
  Consider a marking of $T(Y)$ such  that $v=(1,0)$, see Notation~\ref{nt:markings}\ref{nt:markings:00}.
  By Corollary~\ref{cr:basis-edge}, the vertices adjacent to $v$ are exactly the classes comprising a basis of $\ZZ^2$ with $(1,0)$, i.e., the classes $u_a:=(a,1)$, $a\in\ZZ$.
  By Lemma~\ref{lm:markings}(i), the two triangles containing the edge $(v,u_a)$ have opposite vertices $u_a\pm(1,0)$, i.e., they are $\Delta_{a-1}$ and $\Delta_a$.
\end{proof}

\begin{proposition}\label{pr:vertex-fibration}
  Let $v$ be a vertex of $T(Y)$.
  Given a triangle $\Delta$ of $T(Y)$ containing $v$, denote by $(X_\Delta,D_\Delta)$ the corresponding triangle completion of $Y$, by $L_v^\Delta\subseteq D_\Delta$ the component defining $v$, and by $\pi_v^\Delta\colon X_\Delta\to\PP^1$ the $\PP^1$-fibration defined by the linear system $|D_\Delta-L_v^\Delta|$, see Lemma~\ref{lm:fibrations}.
  Then, for any two triangles $\Delta$ and $\Delta'$ of $T(Y)$ containing $v$, the restrictions $\pi_v^\Delta|_Y$ and $\pi_v^{\Delta'}|_Y$ define the same fibration of $Y$.
\end{proposition}
\begin{proof}
  By Lemma~\ref{lm:vertex-fan}, we may assume that $\Delta=(v,u,w)$ and $\Delta'$ share an edge containing $v$, say $e:=(v,u)$.
  Contraction, say $\sigma$, of the component $L_w^\Delta$ yields a $(0,0)$-completion $(X_e,C_v+C_u)$ of $Y$, where $C_u, C_v$ correspond to $u,v$ respectively.
  Since $\sigma(L_w^\Delta)$ is a node of $C_v+C_u$, we have $\sigma^*|C_u|=|D_\Delta-L_v^\Delta|$ as in the proof of Lemma~\ref{lm:fibrations}, so $\pi_v^\Delta=\pi_e\circ\sigma$ up to an isomorphism of the base, where $\pi_e\colon X_e\to\PP^1$ is the $\PP^1$-fibration with fiber $C_u$.
  As $\sigma$ restricts to the identity on $Y$, we get $\pi_v^\Delta|_Y=\pi_e|_Y$, and the same argument for $\Delta'$ gives $\pi_v^{\Delta'}|_Y=\pi_e|_Y$.
\end{proof}

\begin{definition}\label{def:vertex-fibration}
  Let $v$ be a vertex of $T(Y)$.
  The \emph{fibration attached to $v$} is $\pi_v:=\pi_v^\Delta|_Y$ for a triangle $\Delta$ of $T(Y)$ containing $v$; by Proposition~\ref{pr:vertex-fibration}, it is well defined up to an isomorphism of the base.
\end{definition}

\begin{corollary}\label{cr:vertex-fibration-properties}
  In the notation of Proposition~\ref{pr:vertex-fibration}, let $v$ be a vertex of $T(Y)$ and let $\Delta=(v,u,w)$ be a triangle of $T(Y)$ containing $v$. Then the following assertions hold.
  \begin{enumerate}[(i)]
  \item The divisor $D_\Delta-L_v^\Delta$ is a fiber of $\pi_v^\Delta$, and $\pi_v$ maps $Y$ surjectively onto the complement $\AA^1\subseteq\PP^1$ of the image of this fiber.\label{cr:vertex-fibration-properties:image}
  \item A general fiber of $\pi_v^\Delta$ meets $D_\Delta$ exactly in two distinct points of $L_v^\Delta$; in particular, a general fiber of $\pi_v$ is isomorphic to $\PP^1$ with two points removed.\label{cr:vertex-fibration-properties:fibers}
  \item If $\pi_v$ and $\pi_{v'}$ define the same fibration of $Y$ for vertices $v$ and $v'$ of $T(Y)$, then $v=v'$.\label{cr:vertex-fibration-properties:injective}
  \item Every $g\in\Aut(Y)$ maps the fibers of $\pi_v$ onto the fibers of $\pi_{g(v)}$, where $g(v)$ is the image of $v$ under the action of Theorem~\ref{th:aut-action}; in other words, $\pi_{g(v)}\circ g$ and $\pi_v$ define the same fibration.\label{cr:vertex-fibration-properties:equivariance}
  \end{enumerate}
\end{corollary}
\begin{proof}
  Write $L_v=L_v^\Delta$, $L_u=L_u^\Delta$, $L_w=L_w^\Delta$ for the components of $D_\Delta$.

  \ref{cr:vertex-fibration-properties:image}
  The divisor $L_u+L_w$ belongs to the pencil $|D_\Delta-L_v|$ defining $\pi_v^\Delta$, hence is a fiber, say over $b_\infty\in\PP^1$. Since $(L_u+L_w)\cdot L_v=2$, the component $L_v$ is not contained in a fiber, so every fiber over $b\neq b_\infty$ meets $Y$, and $\pi_v(Y)=\PP^1\setminus\{b_\infty\}$.

  \ref{cr:vertex-fibration-properties:fibers}
  By Lemma~\ref{lm:fibrations}, a general fiber $F$ of $\pi_v^\Delta$ satisfies $F\cap D_\Delta=F\cap L_v$ and $F\cdot L_v=2$. The two points of $F\cap L_v$ are distinct since the degree-2 morphism $\pi_v^\Delta|_{L_v}$ is generically unramified, $\KK$ being of characteristic zero.

  \ref{cr:vertex-fibration-properties:injective}
  Realize $v$ and $v'\neq v$ by distinct components $V$ and $V'$ of a completion obtained from $(X_\Delta,D_\Delta)$ by inner blowups. By \ref{cr:vertex-fibration-properties:fibers}, the extension of $\pi_v$ to this completion contracts $V'$ to a point, whereas $\pi_{v'}^{\Delta'}$ is non-constant on $L_{v'}^{\Delta'}$ by $(D_{\Delta'}-L_{v'}^{\Delta'})\cdot L_{v'}^{\Delta'}=2$; this is a contradiction, since the birational map of completions extending $\id_Y$ identifies the generic points of $V'$ and $L_{v'}^{\Delta'}$, see Remark~\ref{rm:inner-identification}.

  \ref{cr:vertex-fibration-properties:equivariance}
  By Theorem~\ref{th:aut-action} and Lemma~\ref{lm:completion-unique}, $g$ extends to an isomorphism $\tilde g\colon(X_\Delta,D_\Delta)\to(X_{g(\Delta)},D_{g(\Delta)})$ of triangle pairs with $\tilde g(L_v^\Delta)=L_{g(v)}^{g(\Delta)}$. Then $\tilde g^*|D_{g(\Delta)}-L_{g(v)}^{g(\Delta)}|=|D_\Delta-L_v^\Delta|$, and restricting to $Y$ yields the assertion.
\end{proof}

The next proposition characterizes the fibrations of Lemma~\ref{lm:fibrations} among all $\PP^1$-fibrations of a triangle completion.

\begin{proposition}\label{pr:fibrations-characterization}
  Let $(X,D)$ be a triangle completion of $Y$ with $D=L_1+L_2+L_3$ and let $\pi\colon X\to\PP^1$ be a $\PP^1$-fibration.
  Then the following conditions are equivalent:
  \begin{enumerate}[(i)]
  \item $\pi$ is the same fibration as $\pi_i$ of Lemma~\ref{lm:fibrations} for some $i\in\{1,2,3\}$;\label{pr:fibrations-characterization:vertex}
  \item some fiber of $\pi$ is supported on $\supp D$;\label{pr:fibrations-characterization:fiber}
  \item a general fiber of $\pi$ meets exactly one component of $D$;\label{pr:fibrations-characterization:one-component}
  \item $\pi(Y)\neq\PP^1$.\label{pr:fibrations-characterization:image}
  \end{enumerate}
\end{proposition}

\begin{proof}
  Denote by $F$ the class of a fiber of $\pi$.
  For a complete irreducible curve $C\subseteq X$ we have $C\cdot F\geq0$, with equality if and only if $\pi(C)$ is a point; in particular, $F^2=0$.
  Moreover, all fibers of $\pi$ are connected: in the Stein factorization $X\to B'\to\PP^1$ the finite morphism $B'\to\PP^1$ is birational, a general fiber being irreducible, hence an isomorphism onto $\PP^1$.

  \ref{pr:fibrations-characterization:vertex}$\Rightarrow$\ref{pr:fibrations-characterization:one-component},\ref{pr:fibrations-characterization:image} is Corollary~\ref{cr:vertex-fibration-properties}\ref{cr:vertex-fibration-properties:image},\ref{cr:vertex-fibration-properties:fibers}, and \ref{pr:fibrations-characterization:image}$\Rightarrow$\ref{pr:fibrations-characterization:fiber} holds since the fiber over any point of $\PP^1\setminus\pi(Y)$ does not meet $Y$.

  \ref{pr:fibrations-characterization:one-component}$\Rightarrow$\ref{pr:fibrations-characterization:fiber}:
  if a general fiber meets only $L_i$, then $L_j\cdot F=L_k\cdot F=0$. So, $\pi$ contracts $L_j$ and $L_k$, and their images coincide since $L_j\cap L_k\neq\emptyset$.
  Hence some fiber equals $F_0=bL_j+cL_k+R$ with $b,c\geq1$ and $R$ effective without components $L_j,L_k$.
  From $F_0\cdot(L_j+L_k)=0$ and $(bL_j+cL_k)\cdot(L_j+L_k)=0$ we get $R\cdot(L_j+L_k)=0$, so $R$ is disjoint from $L_j\cup L_k$ and $R=0$ by the connectedness of $F_0$.

  \ref{pr:fibrations-characterization:fiber}$\Rightarrow$\ref{pr:fibrations-characterization:vertex}:
  let $F_0=aL_1+bL_2+cL_3$ be a fiber; then $F_0\cdot L_i=0$ whenever the coefficient at $L_i$ is nonzero.
  A single nonzero coefficient gives $F_0^2=-a^2\neq0$, and three nonzero ones give $-a+b+c=a-b+c=a+b-c=0$, i.e., $a=b=c=0$; hence exactly two are nonzero, say $F_0=bL_j+cL_k$.
  Every fiber $G$ of $\pi_i$ is a member of $|L_j+L_k|$, so $G\cdot F=0$ and $\pi$ contracts each component of the connected curve $G$; thus $\pi$ is constant on the fibers of $\pi_i$ and factors as $\pi=\beta\circ\pi_i$ for a morphism $\beta\colon\PP^1\to\PP^1$.
  Since a general fiber of $\pi$ is irreducible, while the fibers of $\pi_i$ are nonempty and pairwise disjoint, $\deg\beta=1$, i.e., $\pi$ and $\pi_i$ define the same fibration.
\end{proof}

\begin{corollary}\label{cr:vertex-fibration-characterization}
  The fibrations of $Y$ attached to the vertices of $T(Y)$ are exactly the restrictions to $Y$ of the $\PP^1$-fibrations $\pi\colon X\to\PP^1$ of triangle completions $(X,D)$ of $Y$ such that $\pi(Y)\neq\PP^1$.
\end{corollary}
\begin{proof}
  Follows from Proposition~\ref{pr:fibrations-characterization}(i,iv).
\end{proof}

\begin{remark}\label{rm:interior-conics}
  The condition $\pi(Y)\neq\PP^1$ in Corollary~\ref{cr:vertex-fibration-characterization} cannot be dropped.
  For example, consider a smooth cubic surface $X\subseteq\PP^3$ and a hyperplane section $D=L_1+L_2+L_3$ that is a triangle of lines, and let $L$ be one of the $24$ lines on $X$ not contained in $\supp D$.
  Then $L$ meets exactly one component of $D$, say $L_3$, since $L\cdot D=1$ for the hyperplane section $D$.
  The planes containing $L$ cut out on $X$ a pencil of conics (hyperplane sections minus $L$). It defines a $\PP^1$-fibration $\pi\colon X\to\PP^1$ whose general fiber $F$ satisfies
  \[
  F\cdot L_1=F\cdot L_2=1,\qquad F\cdot L_3=0.
  \]
  Since $F$ meets two components of $D$, Proposition~\ref{pr:fibrations-characterization} shows that $\pi$ is not attached to a vertex of $T(Y)$ and $\pi|_Y$ is surjective onto $\PP^1$.
\end{remark}

We conclude this section with an intrinsic characterization of the fibrations attached to the vertices of $T(Y)$: they are exactly the fibrations of $Y$ over an affine base whose general fibers are rational curves.
In particular, this description does not refer to a completion of $Y$.
We start with a standard observation.

\begin{lemma}\label{lm:units}
  Every invertible regular function on $Y$ is constant.
\end{lemma}
\begin{proof}
  Let $(X,D)$ be a triangle completion of $Y$ with $D=L_1+L_2+L_3$ and let $f$ be an invertible regular function on $Y$.
  Since $f$ has neither zeros nor poles on $Y=X\setminus\supp D$, we have $\operatorname{div}_X(f)=aL_1+bL_2+cL_3$ for some $a,b,c\in\ZZ$.
  A principal divisor is Cartier and has degree zero on every complete curve, so $\operatorname{div}_X(f)\cdot L_i=0$ for $i=1,2,3$, i.e., $-a+b+c=a-b+c=a+b-c=0$, whence $a=b=c=0$ as in the proof of Proposition~\ref{pr:fibrations-characterization}.
  Thus $f$ is an invertible regular function on the connected projective surface $X$, hence constant.
\end{proof}

\begin{proposition}\label{pr:affine-fibrations}
  Let $\pi\colon Y\to B$ be a surjective morphism onto a smooth affine curve whose general fiber is isomorphic to $\AA^1$ or to $\AA^1\setminus\{0\}$.
  Then $B\cong\AA^1$, a general fiber of $\pi$ is isomorphic to $\AA^1\setminus\{0\}$, and $\pi$ and $\pi_v$ define the same fibration of $Y$ for a unique vertex $v$ of $T(Y)$.
  Consequently, $v\mapsto\pi_v$ is a bijection between the vertices of $T(Y)$ and the fibrations of $Y$ as above, considered up to an isomorphism of the base.
\end{proposition}

\begin{proof}
  Fix a triangle completion $(X,D)$ of $Y$ with $D=L_1+L_2+L_3$.
  We may assume $X$, and hence $Y$, to be smooth: the minimal resolution of $(X,D)$ is a triangle completion of the minimal resolution $r\colon\widetilde Y\to Y$ (as in the proof of Proposition~\ref{pr:triangle-surface-anticanonical} below), and $r$ identifies $T(\widetilde Y)$ with $T(Y)$ and the fibrations $\pi$, $\pi_v$ with $\pi\circ r$, $\pi_v\circ r$, mapping a general fiber of $\pi\circ r$ isomorphically onto a general fiber of $\pi$.

  The curve $B$ is rational, being the image of the rational surface $Y$; since a nonconstant invertible function on $B$ would pull back to one on $Y$, contradicting Lemma~\ref{lm:units}, we get $B\cong\AA^1$ and identify $\overline B\cong\PP^1$ with $\AA^1\cup\{\infty\}$.

  Let $\sigma\colon(\hat X,\hat D)\to(X,D)$ be a minimal composition of blowups resolving the indeterminacy of the rational map $X\dashrightarrow\PP^1$ defined by $\pi$; its centers lie over $\supp D$, so $\hat X\setminus\supp\hat D=Y$, and the resulting morphism $\hat\pi\colon\hat X\to\PP^1$ is a $\PP^1$-fibration with $\hat\pi|_Y=\pi$ and $\hat\pi(Y)=\AA^1$.
  Recall that a component of $\hat D$ is \emph{vertical} if it is contained in a fiber of $\hat\pi$, and \emph{horizontal} otherwise.

  \begin{claim}\label{cl:two-inner-points}
    A general fiber $F$ of $\hat\pi$ meets $\supp\hat D$ in exactly two points, which lie on horizontal components of $\hat D$ that are inner components of $Y$.
  \end{claim}
  \begin{claimproof}
    Since $(K_X+D)\cdot L_i=0$ for each $i$, the Riemann--Roch theorem on the rational surface $X$ gives $\chi(K_X+D)=1$, while $h^2(K_X+D)=h^0(-D)=0$; hence $K_X+D\sim\Theta$ for an effective divisor $\Theta$.
    
    Recall that for the blowup $\mu\colon(X_1,D_1)\to(X_0,D_0)$ of an NC-pair at a boundary point, with exceptional curve $E$, we have $K_{X_1}+D_1=\mu^*(K_{X_0}+D_0)+\delta E$, where $\delta=0$ at a node and $\delta=1$ otherwise.
    By induction along $\sigma$, applying this at each step, we derive that every exceptional component that is not inner for $Y$ appears in $\hat\Theta:=\sigma^*\Theta+\sum\delta_kE_k$ with coefficient at least $1$. 
    Thus, $K_{\hat X}+\hat D\sim\hat\Theta$ for an effective divisor $\hat\Theta$ whose support contains every component of $\hat D$ that is not an inner component of $Y$.

    Since $F\cong\PP^1$ and $F^2=0$, adjunction gives $\hat D\cdot F=(K_{\hat X}+\hat D)\cdot F+2=\hat\Theta\cdot F+2\geq2$, the general fiber $F$ being no component of $\hat\Theta$.
    On the other hand, $F$ is disjoint from the vertical components of $\hat D$ and meets every horizontal component $H$ transversally, since the finite morphism $\hat\pi|_H$ is generically unramified.
    So, $\hat D\cdot F$ equals the number of points of $F\cap\supp\hat D$, which is at most $2$ by the assumption on $\pi$.
    Hence $F$ meets $\supp\hat D$ in exactly two points and $\hat\Theta\cdot F=0$, so every component of $\hat\Theta$ is vertical and every horizontal component of $\hat D$ is inner for $Y$.
  \end{claimproof}

  Thus a general fiber of $\pi$ is isomorphic to $\AA^1\setminus\{0\}$, and a general fiber of $\hat\pi$ meets $\supp\hat D$ along one or two inner components of $Y$, and nothing else.
  Then $\sigma$ is a composition of inner blowups of $(X,D)$, and the triangulation $T(\hat D)$ is, by Lemma~\ref{lm:complex-T-realization}\ref{lm:complex-T-realization:sub}, a finite subcomplex of $T(Y)$.
  Now we may forget about the initial completion $(X,D)$ and work with the NC-completion $(\hat X,\hat D)$ of $Y$ and the $\PP^1$-fibration $\hat\pi$ of $\hat X$ extending $\pi$.

  \begin{claim}\label{cl:vertex-fibration-extends}
    Let $S$ be a component of $\hat D$ and let $s$ be the vertex of $T(\hat D)$ it defines.
    Then $\pi_s$ extends to a $\PP^1$-fibration $\hat\pi_s\colon\hat X\to\PP^1$, and $S$ is the only component of $\hat D$ not contracted by $\hat\pi_s$.
  \end{claim}
  \begin{claimproof}
    It is enough to consider the triangle completion $(X_\Delta,D_\Delta)$ defined by a triangle $\Delta$ of $T(\hat D)$ containing $s$.
    Indeed, $(\hat X, \hat D)$ is obtained from $(X_\Delta,D_\Delta)$ by a sequence of inner blowups, so the $\PP^1$-fibration $\pi_s$ of $(X_\Delta,D_\Delta)$ extends to a $\PP^1$-fibration of $X_\Delta$, all exceptional components being vertical.
  \end{claimproof}

  \begin{claim}\label{cl:single-horizontal}
    The horizontal part of $\hat D$ is a single component $S$, and $\pi$ and $\pi_s$ define the same fibration of $Y$, where $s$ is the vertex of $T(Y)$ defined by $S$.
  \end{claim}
  \begin{claimproof}
    By Claim~\ref{cl:two-inner-points}, the divisor $\hat D$ has a horizontal component $S$; let $\hat\pi_s$ be the corresponding $\PP^1$-fibration and  $G$ be its general fiber. 
    Then $C\cdot G=0$ for every component $C\neq S$ of $\hat D$.
    Since $\hat\pi(Y)=\AA^1$, the fiber $\hat\pi^*(\infty)$ is supported on $\supp\hat D$, and its components are vertical, hence different from $S$.
    So a general fiber $F$ of $\hat\pi$ satisfies $F \cdot G = 0$.
    Then linear systems $|F|$ and $|G|$ coincide, and the claim follows.
  \end{claimproof}

   The final bijectivity claim follows from Corollary~\ref{cr:vertex-fibration-properties}: each $\pi_v$ maps $Y$ onto $\AA^1$ with general fiber isomorphic to $\AA^1\setminus\{0\}$ by \ref{cr:vertex-fibration-properties:image},\ref{cr:vertex-fibration-properties:fibers}, and the uniqueness of $v$ is \ref{cr:vertex-fibration-properties:injective}.
\end{proof}

\begin{remark}\label{rm:affine-fibrations-sharpness}
  The hypothesis that the base is affine cannot be dropped: the interior conic pencils of Remark~\ref{rm:interior-conics} restrict to fibrations $Y\to\PP^1$ whose general fibers are isomorphic to $\AA^1\setminus\{0\}$, but which are attached to no vertex of $T(Y)$.
  Note also that the construction of $\hat\pi$ and Claim~\ref{cl:two-inner-points} of the proof do not use the affineness of the base and show that $Y$ admits no fibration over a smooth curve whose general fiber is isomorphic to $\AA^1$.
\end{remark}

\section{Triangle surfaces}\label{sec:triangle-surfaces}
In this section we restrict to affine surfaces and show that affine triangle surfaces are in fact cubic surfaces.

\begin{definition}\label{def:triangle-surface}
A normal affine surface $Y$ is called a \textit{triangle surface} if it admits a completion $(X,D)$ that is a triangle pair.
\end{definition}

\begin{remark}\label{rm:affine-boundary}
  Let $(X,D)$ be an NC-completion of an affine surface $Y$ and $C\subset X$ be a complete irreducible curve with $C\not\subseteq\supp D$.
  Since $Y$ contains no complete curves, $D\cdot C\geq 1$.
\end{remark}

\begin{proposition}\label{pr:triangle-surface-anticanonical}
  Assume that $(X,D)$ is a triangle completion of a triangle surface $Y$.
  Then $X$ is a (possibly singular) del Pezzo surface of degree 3, and 
  the divisor $D$ is anticanonical.
\end{proposition}
\begin{proof}
  Write $D=L_1+L_2+L_3$. Since $\Gamma(D)$ is a $3$-cycle of $(-1)$-curves, we have $L_i^2=-1$ and $L_i\cdot L_j=1$ for $i\neq j$, so that
  \[
    D^2=3,\qquad -K_X\cdot L_i=1\quad(i=1,2,3),\qquad (K_X+D)\cdot L_i=0,
  \]
  the second equality by the adjunction formula $(K_X+L_i)\cdot L_i=-2$. In particular $(K_X+D)\cdot D=0$.

  Let us show that the divisor $D$ is ample and Cartier.
  As $\supp D$ is disjoint from the singular locus of $X$, the divisor $D$ is Cartier. We verify the Nakai--Moishezon criterion. We have $D^2=3>0$. Let $C\subset X$ be an irreducible curve. If $C=L_i$ for some $i$, then $D\cdot C=1>0$. Otherwise $D\cdot C\geq1$ by Remark~\ref{rm:affine-boundary}. Thus $D$ is ample.

  In the rest of the proof we show that $K_X+D\sim0$.
  By Lemma~\ref{lm:fibrations}, the surface $X$ admits a $\PP^1$-fibration over $\PP^1$, hence $X$ is rational. Let $\rho\colon\widetilde X\to X$ be a resolution of singularities. Since $\rho$ is an isomorphism over $X_{\reg}\supseteq\supp D$, the total transform $\widetilde D:=\rho^{-1}(D)=\widetilde L_1+\widetilde L_2+\widetilde L_3$ is again a cycle of three $(-1)$-curves on the smooth rational surface $\widetilde X$, with the same intersection numbers; in particular $(K_{\widetilde X}+\widetilde D)\cdot\widetilde L_i=0$ for each~$i$.

  By Riemann--Roch on the smooth rational surface $\widetilde X$,
  \[
    \chi\bigl(K_{\widetilde X}+\widetilde D\bigr)=\chi(\O_{\widetilde X})+\tfrac12\,(K_{\widetilde X}+\widetilde D)\cdot\widetilde D=1,
  \]
  and $h^2(K_{\widetilde X}+\widetilde D)=h^0(-\widetilde D)=0$ by Serre duality, since $\widetilde D$ is effective and nonzero. Hence $h^0(K_{\widetilde X}+\widetilde D)\geq1$, and we may choose an effective divisor $M\in|K_{\widetilde X}+\widetilde D|$.

  Let us show that $M\cap\widetilde D=\varnothing$. Write $M=\sum_ia_i\widetilde L_i+B$, where $a_i\geq0$ and the effective divisor $B$ contains no $\widetilde L_i$. Since $\widetilde D\cdot\widetilde L_i=1$ for each $i$ and $M\cdot\widetilde D=(K_{\widetilde X}+\widetilde D)\cdot\widetilde D=0$, we obtain
  \[
    \sum\nolimits_ia_i=\Bigl(\sum\nolimits_ia_i\widetilde L_i\Bigr)\cdot\widetilde D=-B\cdot\widetilde D\leq0,
  \]
  so $a_i=0$ for all $i$ and $B\cdot\widetilde D=0$. Now each $B\cdot\widetilde L_i\geq0$ vanishes, and $M=B$ shares no component with $\widetilde D$; hence $M\cap\widetilde D=\varnothing$.

  Pushing forward, $\rho_*M$ is an effective divisor with $\rho_*M\sim\rho_*(K_{\widetilde X}+\widetilde D)=K_X+D$. Since $M$ is disjoint from $\widetilde D$ and $\rho$ contracts its exceptional curves to points of $Y$, the support of $\rho_*M$ is contained in $Y$. As $Y$ contains no complete curve (Remark~\ref{rm:affine-boundary}), we get $\rho_*M=0$ and $K_X+D\sim0$.

  So, we have $-K_X\sim D$, and the class $-K_X$ is ample. Hence $X$ is a del Pezzo surface of degree $(-K_X)^2=D^2=3$, and $D=L_1+L_2+L_3$ is anticanonical.
\end{proof}

\begin{corollary}\label{cr:triangle-cubic}
  A triangle surface $Y$ is isomorphic to a surface $xyz=f$ in $\AA^3$ for some $f\in\KK[x,y,z]$ of degree at most two.
\end{corollary}
\begin{proof}
  By Proposition~\ref{pr:triangle-surface-anticanonical}, the anticanonical map $\phi_{|-K_X|}$ of the triangle completion $(X,D)$ is an embedding into $\PP^3$ that maps the triangle $D$ of $(-1)$-curves to a triangle of lines.
    Then we may choose homogeneous coordinates $(x:y:z:t)$ on $\PP^3$ so that $\phi_{|-K_X|}(D)$ is the hyperplane section $H=\{t=0\}\cap \phi(X)$, 
    and lines $\phi(L_1),\phi(L_2),\phi(L_3)$ are given  in $H$ by equations $x=0$, $y=0$, $z=0$ respectively.
    Since we have $H^2=D^2=3$, $\phi_{|-K_X|}(X)$ is a cubic surface in $\PP^3$, and the assertion follows from \cite[Theorem~8.3.2]{Dolg_CAG12}.
\end{proof}

\section{Cubic surfaces}\label{sec: cubic}

\begin{lemma}
\label{lemma: after_affine_transform}
Let $Y\subset\AA^3$ be the surface defined by an equation
\begin{equation} \label{eq: cubic equation}
xyz = f(x,y,z),
\end{equation}
where $f$ is a polynomial of degree at most two.
Then, after a suitable affine transformation of $\AA^3$, the equation~\eqref{eq: cubic equation} takes the form
\begin{equation}
\label{eq: delta_equation}
xyz = \delta_x x^2 + \delta_y y^2 + \delta_z z^2 + ax + by + cz + d
\end{equation}
with $\delta_x, \delta_y, \delta_z \in \{0,1\}$.
Moreover, letting $X:=\overline{Y}\subset\PP^3$ and $D:=X\setminus Y$, if $(X,D)$ is a triangle pair, then $\delta_x = \delta_y = \delta_z = 1$, i.e.\ the equation~\eqref{eq: cubic equation} takes the form
\begin{equation}
\label{eq: general_equation}
xyz = x^2 + y^2 + z^2 + ax + by + cz + d.
\end{equation}
\end{lemma}
\begin{proof}
By shifting the origin we get rid of the mixed terms of $f$, and after rescaling the variables we arrive at the equation~\eqref{eq: delta_equation} with $\delta_x, \delta_y, \delta_z \in \{0,1\}$. The surface $X$ is given by the homogeneous equation
\begin{equation*}
xyz = t(\delta_x x^2 + \delta_y y^2 + \delta_z z^2) + t^2 (ax + by + cz) + t^3 d,
\end{equation*}
and the equations of its singular locus at infinity take the form $xy=xz=yz=0$ and $\delta_x x^2+ \delta_y y^2+\delta_z z^2=0$. 

Assume now that $(X,D)$ is a triangle pair; this property is preserved by the above affine transformations. 
Then $D$ lies in the regular locus of $X$ by definition. 
Either of $\delta_x, \delta_y, \delta_z$ being equal to zero gives a singular point on $D$. Hence $\delta_x = \delta_y = \delta_z =1$.
\end{proof}

By Corollary~\ref{cr:triangle-cubic}, every triangle surface is isomorphic to a surface~\eqref{eq: cubic equation}, and its triangle completion is the closure of that surface in $\PP^3$; so by Lemma~\ref{lemma: after_affine_transform} its equation can be brought to the form~\eqref{eq: general_equation}. Conversely, Proposition~\ref{prop: NC_is_triangle} below shows that every surface~\eqref{eq: general_equation} is a triangle surface. Together these identify the class of triangle surfaces with the class of such cubics.

From now on we assume that $Y$ is defined by the equation~\eqref{eq: general_equation}. Its completion $X$ in $\PP^3$ is given by the corresponding homogeneous equation
\begin{equation}
\label{eq: general_equation_homogeneous}
xyz = t(x^2 + y^2 + z^2) + t^2 (ax + by + cz) + t^3 d
\end{equation}
with respect to homogeneous coordinates $(x:y:z:t)$ in $\PP^3$. The boundary divisor $D=X\backslash Y$ consists of three lines
\begin{align*}
L_x = \{x = t = 0\}, \\
L_y = \{y = t = 0\}, \\
L_z = \{z = t = 0\}.
\end{align*}

\begin{lemma}\label{lm:cubic-normal}
  The surface $X\subset\PP^3$ given by \eqref{eq: general_equation_homogeneous} is irreducible and normal.
  Its singular locus is finite and does not meet the plane $\{t=0\}$.
\end{lemma}
\begin{proof}
  Consider the polynomial $F=xyz-t(x^2+y^2+z^2)-t^2(ax+by+cz)-t^3d$ from the equation~\eqref{eq: general_equation_homogeneous}.

  We first treat the plane at infinity. On $\{t=0\}$ the four partial derivatives of $F$ reduce to
  \begin{equation*}
  \frac{\partial F}{\partial x}=yz,\qquad
  \frac{\partial F}{\partial y}=xz,\qquad
  \frac{\partial F}{\partial z}=xy,\qquad
  \frac{\partial F}{\partial t}=-(x^2+y^2+z^2),
  \end{equation*}
  so a singular point of $X$ on $\{t=0\}$ would satisfy $xy=xz=yz=0$ together with $x^2+y^2+z^2=0$, which forces $x=y=z=0$. There is no such point of $\PP^3$, so $X$ has no singularities at infinity.

  Away from $\{t=0\}$ we bound the singular locus using only the first three partial derivatives; this suffices, since their common zero locus contains the singular locus $X_{\sing}$.
  The Gr\"obner basis of the ideal $\left(\frac{\partial F}{\partial x},\frac{\partial F}{\partial y},\frac{\partial F}{\partial z}\right)$ with respect to the lexicographic order $x\succ y\succ z\succ t\succ a\succ b\succ c$ in the polynomial ring $\QQ[x,y,z,t,a,b,c]$, in which $a,b,c$ are treated as variables, contains a polynomial $P$ with leading monomial $z^5t$ and leading coefficient $1$.
  In particular $P$ involves neither $x$ nor $y$, and, being monic, it stays nonzero under any specialization of $a,b,c$ to scalars of $\KK$.
  Since the ideal is homogeneous in $x,y,z,t$, so is $P$, and a nonzero homogeneous polynomial in the two variables $z,t$ has finitely many zeros in $\PP^1$; as $t\neq0$ in the region under consideration, the coordinate $\frac{z}{t}$ admits only finitely many values at singular points of $X$.
  Proceeding in the same way for the two other coordinates, we conclude that $X_{\sing}$ is finite.

  In particular, $X$ is irreducible: otherwise its irreducible components would be of dimension two and would meet in curves, all of whose points are singular on $X$.

  Finally, a hypersurface in $\PP^3$ is Cohen--Macaulay, hence satisfies Serre's condition $S_2$. As its singular locus is finite, it satisfies $R_1$ as well. By Serre's criterion \cite[Chap.~II, Prop.~8.23(b)]{hartshorne1977}, $X$ is normal. 
\end{proof}

\begin{proposition}
\label{prop: NC_is_triangle}
  The NC-pair $(X,D)$ is a triangle completion.
\end{proposition}
\begin{proof}
  By Lemma~\ref{lm:cubic-normal}, the surface $X$ has at most finitely many singular points, none of which lies on the boundary divisor $D=L_x+L_y+L_z$, which is a hyperplane section of $X$.
  The three lines meet pairwise in the three distinct points $(1:0:0:0)$, $(0:1:0:0)$, and $(0:0:1:0)$, and no point lies on all three; hence $D$ has normal crossings and its dual graph $\Gamma(D)$ is a $3$-cycle.
  The line $L_x$ has intersection index 1 with a hyperplane section, hence we have
  \[
  1= L_x\cdot(L_x+L_y+L_z)=L_x^2+2,
  \]
  and $L_x^2=-1$. Similarly, $L_y^2=L_z^2=-1$.
\end{proof}

\begin{notation}
Let us introduce the following involutions on $Y$:
\begin{align*}
\sigma_x: (x, y, z) \mapsto (yz - x - a, y, z), \\
\sigma_y: (x, y, z) \mapsto (x, xz - y - b, z), \\
\sigma_z: (x, y, z) \mapsto (x, y, xy - z - c).
\end{align*}
They are automorphisms of $Y$. 
Indeed, the equation \eqref{eq: general_equation} may be represented as $x(yz-x-a)=y^2+z^2+by+cz+d$, hence $\sigma_x$ belongs to $\Aut(Y)$. 
The same holds for $\sigma_y$ and $\sigma_z$.
Let us also denote by $G_\sigma$ the group generated by $\sigma_x$, $\sigma_y$, and $\sigma_z$. 
\end{notation}

Let us take a closer look at how these involutions act on $X$ as the elements of $\Bir(X, D)$.

\begin{lemma}
\label{lemma: sigma_action}
Let $\widetilde{X}$ be the blowup of $X$ at the point $p_z=L_x\cap L_y$, and $E$ be the exceptional curve. Then the involution $\sigma_z$ is regular on $\widetilde{X}$. It interchanges $L_z$ and $E$ and fixes $L_x$ and $L_y$.
\end{lemma}
\begin{proof}
We have
$$
\sigma_z:
(x:y:z:t)\mapsto (tx:ty:xy-tz-ct^2:t^2)=(x:y:xy/t-z-tc:t),
$$
hence the map fixes $L_x$ and $L_y$, while $L_z$ is sent to $p_z=(0:0:1:0)$. In the affine chart $\{z=1\}$ one can choose $x$ and $y$ to be the local coordinates at the origin $p_z$ which is to be blown up. In the local coordinates the involution reads as
$$
(x, y) \mapsto \frac{t}{xy-t-ct^2}(x, y).
$$
Any point of $L_z$ other than its intersections with $L_x$ and $L_y$ can be obtained as the limit of a curve $x(\tau)=\tau$, $y(\tau)=\alpha \tau$ for some $\alpha\in\KK$ as $\tau\to\infty$. The above shows that the image of this curve under $\sigma_z$ approaches $p_z$ along neither of the coordinate axes, therefore meeting $E$ outside of the points $E\cap L_x$ and $E\cap L_y$. Then a single blow up at $p_z$ is enough to make $\sigma$ a regular automorphism.
\end{proof}

Theorem~\ref{th:aut-action} provides a homomorphism $\mu\colon\Aut(Y)\to\Aut(T(Y))\cong\PGL_2(\ZZ)$, the isomorphism being that of Corollary~\ref{cr:aut-T-PGL}. Lemma~\ref{lemma: sigma_action} determines $\mu$ on the involutions $\sigma_x,\sigma_y,\sigma_z$.

\begin{lemma}\label{lm:sigma-tree-action}
Mark the inner components of $Y$ as in Notation~\ref{nt:markings}, so that the boundary components of the standard triangle completion are $L_x=(1,0)$, $L_y=(0,1)$, and $L_z=(1,1)$. Then
\begin{equation*}
\mu(\sigma_x)=C_x:=\begin{pmatrix}1&0\\2&-1\end{pmatrix},\qquad
\mu(\sigma_y)=C_y:=\begin{pmatrix}-1&2\\0&1\end{pmatrix},\qquad
\mu(\sigma_z)=C_z:=\begin{pmatrix}1&0\\0&-1\end{pmatrix}.
\end{equation*}
\end{lemma}
\begin{proof}
By Remark~\ref{rm:farey}, the vertices of $T(Y)$ adjacent to an edge $\{u,v\}$ are $u+v$ and $u-v$.
The edge $\{L_x,L_y\}=\{(1,0),(0,1)\}$ carries the two triangles with third vertices $L_z=(1,1)$ and $(1,-1)$; since blowing up the node $p_z=L_x\cap L_y$ produces the inner component opposite to $L_z$, the exceptional curve of Lemma~\ref{lemma: sigma_action} is $E=E_z=(1,-1)$.
Applying the same rule to the two other nodes, the exceptional curve over $p_x=L_y\cap L_z$ is $E_x=(1,2)$ and the one over $p_y=L_x\cap L_z$ is $E_y=(2,1)$.

Lemma~\ref{lemma: sigma_action} says that $\sigma_z$ fixes $L_x$ and $L_y$ and interchanges $L_z$ and $E_z$.
An element of $\PGL_2(\ZZ)$ fixing $(1,0)$ and $(0,1)$ is diagonal, and interchanging $(1,1)$ with $(1,-1)$ forces it to be $C_z=\operatorname{diag}(1,-1)$.
By symmetry, $\sigma_x$ fixes $L_y=(0,1)$ and $L_z=(1,1)$ and interchanges $L_x=(1,0)$ with $E_x=(1,2)$: writing $\left(\begin{smallmatrix}p&q\\r&s\end{smallmatrix}\right)$ for a representative, fixing $(0,1)$ gives $q=0$, sending $(1,0)$ to $(1,2)$ gives $r=2p$, and fixing $(1,1)$ then gives $s=-p$, so the class is $C_x$.
The same computation for $\sigma_y$, which fixes $L_x$ and $L_z$ and interchanges $L_y=(0,1)$ with $E_y=(2,1)$, gives $C_y$.
Each of $C_x,C_y,C_z$ is an involution of determinant $-1$, as it must be.
\end{proof}

\begin{proposition}[{\cite[Theorems 1 and 2]{ElHu74}}]
\label{prop: generators_El-Huti}
\begin{enumerate}[(i)]
  \item $G_\sigma = \langle\sigma_x\rangle * \langle\sigma_y\rangle * \langle\sigma_z\rangle$.\label{prop: generators_El-Huti:free}
  \item The group $\Aut(Y)$ is generated by $G_\sigma$ and the linear automorphisms of $Y$.\label{prop: generators_El-Huti:generation}
\end{enumerate}
\end{proposition}

El-Huti works with a projective cubic surface together with a triangle of lines contained in its regular locus, allowing the surface itself to be singular; this coincides with our setting. 

\begin{corollary}\label{cr:sigma-faithful}
The restriction of $\mu$ to $G_\sigma$ is injective. Equivalently, $G_\sigma\cap\Aut^\natural(X,D)=1$, and no nontrivial element of $G_\sigma$ acts trivially on $T(Y)$.
\end{corollary}
\begin{proof}
By Proposition~\ref{prop: generators_El-Huti}\ref{prop: generators_El-Huti:free} the group $G_\sigma$ is the free product $\langle\sigma_x\rangle*\langle\sigma_y\rangle*\langle\sigma_z\rangle$, and by Lemma~\ref{lm:sigma-tree-action} its image under $\mu$ is generated by $C_x,C_y,C_z$. These three involutions generate their free product as well, by the ping-pong lemma \cite[Ch.~II.B]{Harpe00} applied to the three arcs of $\RR\PP^1$ cut out by their eigenlines. So $\mu$ carries a free product of three copies of $\ZZ_2$ onto another, sending free generators to free generators; hence it is injective on $G_\sigma$. The second formulation follows since $\ker\mu=\Aut^\natural(X,D)$ by Theorem~\ref{th:aut-action}.
\end{proof}

\begin{definition}\label{def:monomial}
We call a linear automorphism of $\AA^3$ \emph{monomial} if it permutes the coordinates $x,y,z$ up to nonzero scalar factors.
\end{definition}
\begin{lemma}
\label{lemma: linear_automorphisms}
Every linear automorphism of $Y$ is monomial in the sense of Definition~\ref{def:monomial}, with all three scalar factors equal to $\pm1$ and of product $1$.
\end{lemma}

\begin{proof}
Take a linear transformation $\varphi$ lying in $\Aut(Y)$. Then $\varphi$ then must stabilize the one-dimensional space $\KK\cdot xyz$, meaning it can only permute the variables possibly multiplying them by constants.

Denote, respectively, by $\lambda_1, \lambda_2$, and $\lambda_3$ the nonzero constant by which $\varphi$ multiplies each variable after the underlying permutation. 
Then we should have $\lambda_1\lambda_2\lambda_3=\lambda_1^2=\lambda_2^2=\lambda_3^2$, which leads to $\lambda_1\lambda_2\lambda_3 = 1$, and $\lambda_1, \lambda_2, \lambda_3=\pm 1$.
\end{proof}

Note that an affine automorphism of $Y$ is automatically linear: its linear part is monomial by the argument above, and a nonzero translation would reintroduce mixed quadratic monomials $xy$, $yz$, $zx$, which nothing in the equation can cancel. Equivalently, the shift of Lemma~\ref{lemma: after_affine_transform} that kills the mixed terms is unique.

\begin{remark}
\label{remark: structure_of_PGL}
The homomorphism $\mu$ identifies $G_\sigma$ with the kernel of the reduction map $\PGL_2(\ZZ)\to\PGL_2(\FF_2)$: it is injective on $G_\sigma$ by Corollary~\ref{cr:sigma-faithful}, its image lies in the kernel since $C_x,C_y,C_z\equiv I\pmod 2$, and conversely these three involutions generate the kernel. 

Since $\PGL_2(\FF_2)$ is the group of permutations of the three points of $\PP(\FF_2^2)$, hence isomorphic to $S_3$, and the coordinate permutations split the quotient,
\begin{equation*}
\PGL_2(\ZZ)\cong G_\sigma\rtimes S_3 .
\end{equation*}
Furthermore, put $H := \langle \sigma_z\sigma_x,\ \sigma_z\sigma_y\rangle$, an index-$2$ subgroup of the free product $G_\sigma=\langle\sigma_x\rangle*\langle\sigma_y\rangle*\langle\sigma_z\rangle$ and hence free on the two generators. As $\sigma_z$ normalizes $H$ and together they generate $G_\sigma$, we see that $G_\sigma \cong \ZZ_2 \ltimes H$. Under $\mu$ the subgroup $H$ corresponds to the classes of determinant $1$ in the kernel: these form the subgroup $\overline\Gamma(2)\subset\PSL_2(\ZZ)$, generated freely by $C_zC_x$ and $C_zC_y$ --- the squares of the elementary transformation matrices of Remark~\ref{rm:elementary-moves}\ref{rm:elementary-moves:elementary} for the standard basis --- while $C_z$ represents the remaining coset.
\end{remark}

Let $K_4$ be the Klein four-group consisting of the maps $(x, y, z)\mapsto (\varepsilon_x x, \varepsilon_y y, \varepsilon_z z)$ with $\varepsilon_x\varepsilon_y\varepsilon_z=1$. Then there is an action of the group $S_4= K_4 \rtimes S_3$ on $\AA^3$.

\begin{theorem}
\label{thm: automorphisms_of_cubic_surfaces}
Define the maps
\begin{align*}
\tau &: (x, y, z)\mapsto (y, x, z), \\
\theta &: (x, y, z)\mapsto (-x, -y, z).
\end{align*}
Then any cubic surface $Y$ given by the equation~\eqref{eq: general_equation} can be brought to the form with $(a, b, c)$ from Table~\ref{table_markov} by some element $g\in S_4$ and $\Aut(Y) = G_\sigma \rtimes \Gamma$, where the group $\Gamma$ is a subgroup of $S_4 = K_4 \rtimes S_3$ conjugated to $\Gamma_0$ from the respective row with the element~$g$. All possible $\Gamma_0$ are listed in Table~\ref{table_markov}.

\renewcommand{\arraystretch}{1.3}
\begin{table}[h]
\caption{}\label{table_markov}
\centering
\begin{tabular}{|c|c|}
\hline
Conditions & $\Gamma_0$ \\
\hline
$a = b = c = 0$ & $K_4 \rtimes S_3$ \\
\hline
$a = b = c\neq 0$ & $S_3$ \\
\hline
$a = b = 0\neq c$ & $\langle\theta\rangle_2\times\langle\tau\rangle_2$ \\
\hline
$a = b \neq 0, \pm c\neq a$ & $\langle\tau\rangle_2$ \\
\hline
$a^2, b^2, c^2$ all distinct & $1$ \\

\hline
\end{tabular}
\renewcommand{\arraystretch}{1}
\end{table}


\end{theorem}

\begin{proof}
By Proposition~\ref{prop: generators_El-Huti}\ref{prop: generators_El-Huti:generation} and Lemma~\ref{lemma: linear_automorphisms}, the group $\Aut(Y)$ is generated by $G_\sigma$ and by the group $\mathrm{Lin}(Y)$ of linear automorphisms of $Y$, and the latter is contained in the group $\Gamma$ of signed coordinate permutations: those maps that permute the coordinates $x,y,z$ by some $\pi\in S_3$ and then multiply them by signs whose product equals $1$. Such a map is determined by the pair $(\varepsilon,\pi)\in K_4\rtimes S_3\cong\Gamma $, the normal subgroup $K_4$ being the one of the theorem.
Every element of $\Gamma$ preserves $xyz$, $x^2+y^2+z^2$ and $d$. So,
\begin{equation*}
\mathrm{Lin}(Y)=\operatorname{Stab}_\Gamma(\ell),\qquad \ell=ax+by+cz .
\end{equation*}
Note that $\Gamma\cong S_4$ and the map 
\begin{equation}
\label{eq: repr_isom}
\ell=a x+by+cz\mapsto (a+b+c,-a-b+c, a-b-c, -a+b-c)
\end{equation}
sets up an equivariant isomorpism of the representation of $\Gamma$ in the space of linear forms in variables $x, y, z$ and the standard representation of $S_4$ in $W=\{(x_1,\dots,x_4)\mid x_1+\dots+x_4=0\}\subset\KK^4$. Then this is essentially the problem of calculating stabilizers of points in the $S_4$ in $\KK^4$. The isomorphism class of the stabilizer of a point $x=(x_1,\dots,x_4)\in\KK^4$ is determined by the sizes of groups of equal coordinates $x_i$. These sizes may be $\{4\}, \{3, 1\}, \{2, 2\}, \{2, 1, 1\}, \{1, 1, 1, 1\}$ with stibilizers isomorphic (as abstract groups) to $S_4$, $S_3$, $K_4$, $\ZZ_2$, $1$. By equating coordinates of the image of $\ell$ in~(\ref{eq: repr_isom}) with respect to a chosen partition we obtain the following conditions on the original coordinates $a, b, c$ which should be read up to a permutation:
\begin{itemize}
  \item $a=b=c=0$,
  \item $a^2=b^2=c^2\neq 0$,
  \item $a = b = 0$, $c \neq 0$,
  \item $a^2 = b^2\neq 0, c^2\neq a^2,b^2$,
  \item $a^2, b^2, c^2$ are all distinct.
\end{itemize}

Up to the action of the group $\Gamma$ we get five cases described in the table. If $Y'$ is the variety corresponding to a triple $(a',b',c')$ then there exists an element $g \in \Gamma$ such that $(a',b',c') = g(a,b,c)$, where $(a,b,c)$ satisfies one of the conditions in the table. So the group $\mathrm{Lin}(Y)$ can be obtained from the group in the table by conjugating with $g$.

It remains to assemble the extensions. Conjugation by $\varphi_{\pi,\varepsilon}$ permutes $\sigma_x,\sigma_y,\sigma_z$ according to $\pi$, as one reads off their defining formulas, so $G_\sigma$ is normal in $\Aut(Y)$. Moreover $G_\sigma\cap\mathrm{Lin}(Y)=1$: a linear automorphism permutes the vertices $(1,0),(0,1),(1,1)$ of the base triangle of $T(Y)$, while $\mu(G_\sigma)$ lies in the kernel of the reduction $\PGL_2(\ZZ)\to\PGL_2(\FF_2)$ by Remark~\ref{remark: structure_of_PGL}, so an element of $G_\sigma$ fixes the marking of every vertex modulo~$2$. As the three vertices of the base triangle have distinct nonzero reductions, an element of the intersection fixes all three of them, hence is trivial in $\PGL_2(\ZZ)$, hence is trivial by Corollary~\ref{cr:sigma-faithful}. Thus $\Aut(Y)=G_\sigma\rtimes\mathrm{Lin}(Y)$.
\end{proof}

\begin{remark}\label{cr:aut-natural-cubic}
\begin{enumerate}
 \item Let $Y$ be a triangle surface and $(X,D)$ its triangle completion. Then $\Aut^\natural(X,D)$ is the group of sign changes preserving the equation of $Y$; in particular it is finite of order at most $4$. Explicitly, it is trivial in the second and the last two rows of the table above, isomorphic to $\ZZ_2$ in the third, and isomorphic to $K_4$ in the first.
 \item So in every row of the table in Theorem~\ref{thm: automorphisms_of_cubic_surfaces} the image of $\mu$ is $G_\sigma$ extended by the subgroup of $S_3\cong\PGL_2(\FF_2)$ realized by coordinate permutations.
 \item By Remark~\ref{remark: structure_of_PGL}, the automorphism group of surfaces corresponding to the first and second row of Table~\ref{table_markov} can be written as $K_4\rtimes \PGL_2(\ZZ)$ and  $\PGL_2(\ZZ)$, respectively.
\end{enumerate}
\end{remark}


\begin{remark}\label{rm:double-cover}
The projection $(x,y)\colon Y\to\AA^2$ is a double cover: the equation~\eqref{eq: general_equation} is quadratic in $z$, namely $z^2-(xy-c)z+(x^2+y^2+ax+by+d)=0$. Its deck transformation is the Vieta involution $\sigma_z$, which exchanges the two roots. In particular $\sigma_z$ preserves the fibrations $\pi_{L_x}=x$ and $\pi_{L_y}=y$ and moves $\pi_{L_z}=z$, cf. Lemma~\ref{lm:sigma-tree-action}.
\end{remark}

\section{Examples}\label{sec: examples}

Here we apply our results to some known classes of cubic surfaces.

\subsection{The Markov surface}\label{sec:markov-surface}

The Markov surface is defined in $\AA^3$ by the Markov equation
\begin{equation*}
x^2+y^2+z^2 = 3xyz .
\end{equation*}
Rescaling all three coordinates by $\tfrac13$ brings it to the normal form of Lemma~\ref{lemma: after_affine_transform},
\begin{equation*}
xyz = x^2+y^2+z^2 ,
\end{equation*}
so that $a=b=c=d=0$. This is the first row of the table of Theorem~\ref{thm: automorphisms_of_cubic_surfaces}, whence
\begin{equation*}
\Aut(Y)\cong K_4\rtimes\PGL_2(\ZZ),
\end{equation*}
and $\Aut^\natural(X,D)=K_4$ by Remark~\ref{cr:aut-natural-cubic}(1): the Markov surface is the triangle surface with the largest possible group of automorphisms acting trivially on $T(Y)$. This recovers the description of $\Aut(Y)$ given in~\cite{Pe21} for the Markov-like surfaces $x^2+y^2+z^2-xyz=c$.

Let us make the relation to the classical Markov tree explicit.
Under this rescaling a Markov triple $(m,n,k)$ becomes the integral point $(3m,3n,3k)$ of $Y$, and the classical Vieta involution $(x,y,z)\mapsto(3yz-x,y,z)$ becomes the involution $\sigma_x$, which reads $(x,y,z)\mapsto(yz-x,y,z)$ since $a=0$.
The Markov triples are generated from $(1,1,1)$ by the Vieta involutions together with the coordinate permutations, see~\cite{Aig13}, so the argument at the end of this section applies to the Markov surface as well.

\subsection{Fricke surfaces}
Here we recall notion and constructions from~\cite{UlYi22}.
The cubic surface $F^2$ given by the equation
\begin{equation}
\label{eq:Fricke_equation}
(x+y+z)^2=9xyz
\end{equation}
is called the \textit{double Fricke surface}. In~\cite{UlYi22}, a group law on this surface was introduced, and its connections with some automorphisms of $F^2$ were studied. Let us denote the closure of $F^2$ in $\PP^3$ by $\overline{F}^2$. One quickly verifies that there is only one singular point $(0:0:0:1)$ of $\overline{F}^2$. Besides the three curves of the boundary divisor at infinity, there are three finite lines defined by the parameterizations $\{(s,-s,0)\}$, $\{(0,s,-s)\}$, and $\{(s,0,-s)\}$ in $\AA^3$.

Let us bring the equation~\eqref{eq:Fricke_equation} to the normal form of Lemma~\ref{lemma: after_affine_transform}. Substituting $x\mapsto\tfrac19x+\tfrac29$ and likewise in $y$ and $z$, we obtain
\begin{equation*}
xyz = x^2+y^2+z^2+8(x+y+z)+28 ,
\end{equation*}
so $a=b=c=8$ and $d=28$. By the second row of the table of Theorem~\ref{thm: automorphisms_of_cubic_surfaces},
\begin{equation*}
\Aut(F^2)\cong\PGL_2(\ZZ),
\end{equation*}
and $\Aut^\natural(X,D)=1$ by Remark~\ref{cr:aut-natural-cubic}(1).

The projection $\PP^3\to\PP^2$, $(x:y:z:t)\mapsto (x:y:z)$, restricted to $\overline{F}^2$, is the projection from the singular point $(0:0:0:1)$ and sets up a birational isomorphism between the completion $\overline{F}^2$ of the double Fricke surface and the projective plane $\PP^2$. The inverse reads as
\begin{equation*}
(x:y:z) \mapsto \bigl(L^2 x : L^2 y : L^2 z : 9xyz\bigr),
\end{equation*}
with $L=x+y+z$.

The map
\begin{equation*}
(x:y:z) \mapsto (x^2 Q^2:y^2 Q^2: z^2 Q^2: 9 x^2 y^2 z^2),\qquad Q=x^2+y^2+z^2,
\end{equation*}
also has image in $\overline{F}^2$, but it is the birational inverse above precomposed with the squaring map $(x:y:z)\mapsto(x^2:y^2:z^2)$, hence of degree $4$ rather than $1$. That squaring is not an accident: it is the same squaring that turns a Markov triple $(m,n,k)$ into the positive integral point $(m^2,n^2,k^2)$ of $F^2$, as recalled below.

Introduce the following automorphism of the affine surface $F^2$, the composition of a transposition and $\sigma_z$ (which is of infinite order):
\begin{equation*}
(x,y,z)\mapsto (x, 9xy - 2x - 2y - z, y).
\end{equation*}
By~\cite{UlYi22}, along with permutations of coordinates this mapping generates all positive integral points of $F^2$ from the solution $(1,1,1)$ of the equation~(\ref{eq:Fricke_equation}). It is also stated in \emph{loc. cit.} that these positive integral points are precisely of the form $(m^2,n^2,k^2)$, with $(m, n, k)$ being a Markov triple.

\subsection{Generalized Markov numbers} 
Here we recall definitions from \cite{Gyo21, GyMa23}

The generalized Markov equation is:
\begin{equation*}
x^2 + y^2 +z^2 +xy + xz + yz = 6xyz.
\end{equation*}
Another form of this equation is:
\begin{equation*}
(x+y)^2 + (x+z)^2 +(y+z)^2 = 12xyz.
\end{equation*}
A generalized Markov tuple is a tuple of positive integers $(a, b, c)$ satisfying the generalized Markov equation. In~\cite{BaSe24} the following involution was introduced:
\begin{equation*}
(x,y,z)\mapsto \left(x, y, \frac{x^2 + xy + y^2}{z}\right)
\end{equation*}
or
\begin{equation*}
(x,y,z)\mapsto (x, y, 6xy - x - y - z).
\end{equation*}
It was proved that along with permutations of coordinates this map generates all generalized Markov tuples from the generalized Markov tuple $(1, 1, 1)$.

As with the double Fricke surface, we bring this to the normal form of Lemma~\ref{lemma: after_affine_transform}. Substituting $x\mapsto\tfrac16x+\tfrac16$ and likewise in $y$ and $z$, we obtain
\begin{equation*}
xyz = x^2+y^2+z^2+3(x+y+z)+5 ,
\end{equation*}
so $a=b=c=3$ and $d=5$. The second row of the table of Theorem~\ref{thm: automorphisms_of_cubic_surfaces} gives $\Aut(Y)\cong\PGL_2(\ZZ)$, exactly as for the double Fricke surface.

All three cases of this section repeat the general picture: the positive integral points of the surface are reached from $(1,1,1)$ by involutions together with the coordinate permutations, and the group generated by these maps acts on the induced 3-regular tree of Diophantine solutions is isomorphic to $\PGL_2(\ZZ)$.

\bibliographystyle{plainurl}
\bibliography{markov} 

@misc{Abb24-preprint,
  title = {Unlikely Intersections Problem for Automorphisms of {{Markov}} Surfaces},
  author = {Abboud, Marc},
  date = {2024},
  year = {2024},
  url = {https://arxiv.org/abs/2401.05762}
}

@book{Aig13,
  title = {Markov's Theorem and 100 Years of the Uniqueness Conjecture},
  author = {Aigner, M.},
  date = {2013},
  year = {2013},
  publisher = {Springer}
}

@phdthesis{Bar91,
  title = {The {{Markoff}} Equation and Equations of {{Hurwitz}}},
  author = {Baragar, Arthur},
  date = {1991},
  year = {1991},
  school = {Brown University},
  type = {PhD thesis}
}

@article{Coh55,
  title = {Approach to {{Markoff}}'s Minimal Forms through Modular Functions},
  author = {Cohn, Harvey},
  date = {1955},
  year = {1955},
  journaltitle = {Annals of Mathematics},
  journal = {Annals of Mathematics},
  series = {2},
  volume = {61},
  number = {1},
  pages = {1--12},
  doi = {10.2307/1969618}
}

@book{Dolg_CAG12,
  title = {Classical Algebraic Geometry: A Modern View},
  author = {Dolgachev, Igor V.},
  date = {2012},
  year = {2012},
  publisher = {Cambridge University Press}
}

@book{FrKl12,
  title = {Vorlesungen über Die Theorie Der Automorphen Funktionen. {{Band}} {{II}}: {{Die}} Funktionentheoretischen Ausführungen Und Die Anwendungen},
  author = {Fricke, Robert and Klein, Felix},
  date = {1912},
  year = {1912},
  publisher = {Teubner},
  location = {Leipzig}
}

@book{Seg42,
  title = {The Non-Singular Cubic Surfaces},
  author = {Segre, Beniamino},
  date = {1942},
  year = {1942},
  publisher = {Clarendon Press},
  location = {Oxford}
}

@article{Sil26,
  title = {The {{Markoff}} Equation: {{Past}}, {{Present}}, {{Future}}},
  author = {Silverman, Joseph H.},
  date = {2026},
  year = {2026},
  journaltitle = {Notices of the American Mathematical Society},
  journal = {Notices of the American Mathematical Society},
  volume = {73},
  number = {5},
  pages = {367--375},
  doi = {10.1090/noti3336}
}

@misc{And25,
  title = {Holomorphic Automorphisms of {{Markov-type}} Surfaces},
  author = {Andrist, Rafael B.},
  date = {2025},
  year = {2025},
  url = {https://arxiv.org/abs/2502.15101}
}

@article{BaSe24,
  title = {A Generalization of {{Markov Numbers}}},
  author = {Banaian, Esther and Sen, Archan},
  date = {2024-04},
  year = {2024},
  journaltitle = {The Ramanujan Journal},
  journal = {The Ramanujan Journal},
  volume = {63},
  number = {4},
  pages = {1021--1055},
  issn = {1382-4090, 1572-9303},
  doi = {10.1007/s11139-023-00801-6},
  url = {https://link.springer.com/10.1007/s11139-023-00801-6},
  urldate = {2024-11-23},
  langid = {english}
}

@book{BPV04,
  title = {Compact {{Complex Surfaces}}},
  author = {Barth, Wolf P. and Hulek, Klaus and Peters, Chris A. M. and Van De Ven, Antonius},
  date = {2004},
  year = {2004},
  series = {Ergebnisse Der {{Mathematik}} Und Ihrer {{Grenzgebiete}}. 3. {{Folge}} / {{A Series}} of {{Modern Surveys}} in {{Mathematics}}},
  volume = {4},
  publisher = {Springer Berlin Heidelberg},
  location = {Berlin, Heidelberg},
  doi = {10.1007/978-3-642-57739-0},
  url = {https://link.springer.com/10.1007/978-3-642-57739-0},
  urldate = {2025-04-16},
  isbn = {978-3-540-00832-3 978-3-642-57739-0},
  langid = {english}
}

@article{CaLo09,
  title = {Dynamics on {{Character Varieties}} and {{Malgrange}} Irreducibility of {{Painlevé VI}} Equation},
  author = {Cantat, Serge and Loray, Frank},
  date = {2009},
  year = {2009},
  journaltitle = {Annales de l'Institut Fourier},
  journal = {Annales de l'Institut Fourier},
  volume = {59},
  number = {7},
  pages = {2927--2978},
  doi = {10.5802/aif.2512},
  url = {https://www.numdam.org/articles/10.5802/aif.2512/},
  langid = {english}
}

@article{DeSt24,
  title = {Smooth Limits of Plane Curves of Prime Degree and {{Markov}} Numbers},
  author = {DeVleming, Kristin and Stapleton, David},
  date = {2024-06},
  year = {2024},
  journaltitle = {Journal de l’École polytechnique — Mathématiques},
  journal = {Journal de l’École polytechnique — Mathématiques},
  volume = {11},
  pages = {683--731},
  issn = {2270-518X},
  doi = {10.5802/jep.263},
  url = {https://jep.centre-mersenne.org/articles/10.5802/jep.263/},
  urldate = {2024-11-23},
  langid = {english}
}

@article{ElHu74,
  title = {{{CUBIC SURFACES OF MARKOV TYPE}}},
  author = {Èl'-Huti, M H},
  date = {1974-04-30},
  year = {1974},
  journaltitle = {Mathematics of the USSR-Sbornik},
  journal = {Mathematics of the USSR-Sbornik},
  shortjournal = {Math. USSR Sb.},
  volume = {22},
  number = {3},
  pages = {333--348},
  publisher = {Steklov Mathematical Institute},
  issn = {0025-5734},
  doi = {10.1070/sm1974v022n03abeh001696},
  url = {https://www.mathnet.ru/eng/sm3405},
  urldate = {2025-10-12}
}

@article{GyMa23,
  title = {Generalization of {{Markov Diophantine Equation}} via {{Generalized Cluster Algebra}}},
  author = {Gyoda, Yasuaki and Matsushita, Kodai},
  date = {2023-10-20},
  year = {2023},
  journaltitle = {The Electronic Journal of Combinatorics},
  journal = {The Electronic Journal of Combinatorics},
  shortjournal = {Electron. J. Combin.},
  volume = {30},
  number = {4},
  pages = {P4.10},
  issn = {1077-8926},
  doi = {10.37236/11420},
  url = {https://www.combinatorics.org/ojs/index.php/eljc/article/view/v30i4p10},
  urldate = {2025-10-01}
}

@unpublished{Gyo21,
  title = {Positive Integer Solutions to {{(x+y)²+(y+z)²+(z+x)²=12xyz}}},
  author = {Gyoda, Y.},
  date = {2021},
  year = {2021},
  eprint = {2109.09639},
  eprinttype = {arxiv},
  url = {https://arxiv.org/abs/2109.09639},
}

@book{hartshorne1977,
  title = {Algebraic {{Geometry}}},
  author = {Hartshorne, Robin},
  date = {1977},
  year = {1977},
  series = {Graduate {{Texts}} in {{Mathematics}}},
  volume = {52},
  publisher = {Springer New York},
  location = {New York, NY},
  doi = {10.1007/978-1-4757-3849-0},
  url = {http://link.springer.com/10.1007/978-1-4757-3849-0},
  urldate = {2025-06-27},
  isbn = {978-1-4419-2807-8 978-1-4757-3849-0}
}

@article{HaTe20,
  title = {Solutions of a Generalized Markoff Equation in {{Fibonacci}} Numbers},
  author = {Hashim, Hayder Raheem and Tengely, Szabolcs},
  date = {2020-10},
  year = {2020},
  journaltitle = {Mathematica Slovaca},
  journal = {Mathematica Slovaca},
  volume = {70},
  number = {5},
  pages = {1069--1078},
  issn = {1337-2211, 0139-9918},
  doi = {10.1515/ms-2017-0414},
  url = {https://www.degruyter.com/document/doi/10.1515/ms-2017-0414/html},
  urldate = {2024-11-23},
  langid = {english}
}

@article{JinSch01,
  title = {A {{Diophantine}} Equation Appearing in {{Diophantine}} Approximation},
  author = {Jin, Yuan and Schmidt, Asmus L.},
  date = {2001-12},
  year = {2001},
  journaltitle = {Indagationes Mathematicae},
  journal = {Indagationes Mathematicae},
  shortjournal = {Indagationes Mathematicae},
  volume = {12},
  number = {4},
  pages = {477--482},
  issn = {00193577},
  doi = {10.1016/S0019-3577(01)80036-7},
  url = {https://linkinghub.elsevier.com/retrieve/pii/S0019357701800367},
  urldate = {2025-10-01},
  langid = {english}
}

@online{Kol24-irred,
  title = {Log {{K3}} Surfaces with Irreducible Boundary},
  author = {Kollár, János},
  date = {2024-07-10},
  year = {2024},
  eprint = {2407.08051},
  eprinttype = {arXiv},
  eprintclass = {math},
  doi = {10.48550/arXiv.2407.08051},
  url = {http://arxiv.org/abs/2407.08051},
  urldate = {2025-10-12},
  pubstate = {prepublished}
}

@article{KoVi25,
  title = {Cubic Surfaces with Infinite, Discrete Automorphism Group},
  author = {Kollár, János and Villalobos-Paz, David},
  date = {2025},
  year = {2025},
  journaltitle = {Atti della Accademia Nazionale dei Lincei. Rendiconti Lincei. Matematica e Applicazioni},
  journal = {Atti della Accademia Nazionale dei Lincei. Rendiconti Lincei. Matematica e Applicazioni},
  volume = {36},
  number = {3},
  pages = {563--590},
  doi = {10.4171/RLM/1083},
  url = {https://arxiv.org/abs/2410.03934}
}

@article{Lam16,
  title = {Diophantine Equations via Cluster Transformations},
  author = {Lampe, Philipp},
  date = {2016-09},
  year = {2016},
  journaltitle = {Journal of Algebra},
  journal = {Journal of Algebra},
  shortjournal = {Journal of Algebra},
  volume = {462},
  pages = {320--337},
  issn = {00218693},
  doi = {10.1016/j.jalgebra.2016.04.033},
  url = {https://linkinghub.elsevier.com/retrieve/pii/S0021869316301193},
  urldate = {2025-10-01},
  langid = {english}
}

@article{Mar1879,
  title = {Sur les formes quadratiques binaires indéfinies},
  author = {Markoff, A.},
  date = {1879-09},
  year = {1879},
  journaltitle = {Mathematische Annalen},
  journal = {Mathematische Annalen},
  shortjournal = {Math. Ann.},
  volume = {15},
  number = {3--4},
  pages = {381--406},
  issn = {0025-5831, 1432-1807},
  doi = {10.1007/BF02086269},
  url = {http://link.springer.com/10.1007/BF02086269},
  urldate = {2025-10-01},
  langid = {french}
}

@article{Mar1880,
  title = {Sur les formes quadratiques binaires indéfinies},
  author = {Markoff, A.},
  date = {1880-09},
  year = {1880},
  journaltitle = {Mathematische Annalen},
  journal = {Mathematische Annalen},
  shortjournal = {Math. Ann.},
  volume = {17},
  number = {3},
  pages = {379--399},
  issn = {0025-5831, 1432-1807},
  doi = {10.1007/BF01446234},
  url = {http://link.springer.com/10.1007/BF01446234},
  urldate = {2025-10-01},
  langid = {french}
}

@article{Mor52,
  title = {The Congruence Ax3+by3+c=0 (Mod Xy), and Integer Solutions of Cubic Equations in Three Variables},
  author = {Mordell, L. J.},
  date = {1952-01},
  year = {1952},
  journaltitle = {Acta Mathematica},
  journal = {Acta Mathematica},
  volume = {88},
  pages = {77--83},
  publisher = {Institut Mittag-Leffler},
  issn = {0001-5962, 1871-2509},
  doi = {10.1007/BF02392129},
  url = {https://projecteuclid.org/journals/acta-mathematica/volume-88/issue-none/The-congruence-ax3by3c0-modxy-and-integer-solutions-of-cubic-equations/10.1007/BF02392129.full},
  urldate = {2025-10-01},
  issue = {none}
}

@article{Mor53,
  title = {On the {{Integer Solutions}} of the {{Equation}} X2+y2+z2+2xyz = n},
  author = {Mordell, L. J.},
  date = {1953-10},
  year = {1953},
  journaltitle = {Journal of the London Mathematical Society},
  journal = {Journal of the London Mathematical Society},
  shortjournal = {Journal of the London Mathematical Society},
  volume = {s1-28},
  number = {4},
  pages = {500--510},
  issn = {00246107},
  doi = {10.1112/jlms/s1-28.4.500},
  url = {http://doi.wiley.com/10.1112/jlms/s1-28.4.500},
  urldate = {2025-10-01},
  langid = {english}
}

@article{Pe21,
  title = {Автоморфизмы поверхностей марковского типа},
  author = {Перепечко, Александр Юрьевич},
  date = {2021},
  year = {2021},
  journaltitle = {Математические заметки},
  journal = {Математические заметки},
  volume = {110},
  number = {5},
  pages = {744--750},
  issn = {0025-567X, 2305-2880},
  doi = {10.4213/mzm13263},
  url = {http://mi.mathnet.ru/mzm13263},
  urldate = {2024-11-23},
  langid = {russian}
}

@article{PeZa26,
  title = {Automorphism Groups of Rigid Affine Surfaces: The Identity Component},
  author = {Perepechko, Alexander and Zaidenberg, Mikhail},
  journal = {Algebraic Geometry},
  year = {2026},
  note = {accepted for publication}
}

@misc{ReRo24,
  title = {Questions about the Dynamics on a Natural Family of Affine Cubic Surfaces},
  author = {Rebelo, Julio and Roeder, Roland},
  date = {2024},
  year = {2024},
  url = {https://arxiv.org/abs/2307.10962}
}

@article{Ros79,
  title = {Über Die {{Diophantische Gleichung}} Ax2 + By2 + Cz2 = Dxyz.},
  author = {Rosenberger, Gerhard},
  date = {1979-01-01},
  year = {1979},
  journaltitle = {Journal für die reine und angewandte Mathematik (Crelles Journal)},
  journal = {Journal für die reine und angewandte Mathematik (Crelles Journal)},
  volume = {1979},
  number = {305},
  pages = {122--125},
  issn = {0075-4102, 1435-5345},
  doi = {10.1515/crll.1979.305.122},
  url = {https://www.degruyter.com/document/doi/10.1515/crll.1979.305.122/html},
  urldate = {2025-10-01}
}

@article{UlYi22,
  title = {Composition Laws on the {{Fricke}} Surface and {{Markov}} Triples},
  author = {Uludağ, Abdurrahman Muhammed and Yilmaz, Esra Ünal},
  date = {2022-01},
  year = {2022},
  journaltitle = {Turkish Journal of Mathematics},
  journal = {Turkish Journal of Mathematics},
  volume = {46},
  number = {3},
  pages = {785--799},
  issn = {1300-0098},
  doi = {10.55730/1300-0098.3123},
  url = {https://journals.tubitak.gov.tr/math/vol46/iss3/8},
  urldate = {2024-11-23},
  langid = {english}
}

@incollection{FKS12,
  title = {Asymptotically Rigid Mapping Class Groups and {{Thompson}}'s Groups},
  author = {Funar, Louis and Kapoudjian, Christophe and Sergiescu, Vlad},
  date = {2012},
  year = {2012},
  booktitle = {Handbook of {{Teichm\"uller}} Theory, {{Volume III}}},
  editor = {Papadopoulos, Athanase},
  series = {{{IRMA}} Lectures in Mathematics and Theoretical Physics},
  volume = {17},
  pages = {595--664},
  publisher = {European Mathematical Society},
  location = {Z\"urich}
}

@book{Harpe00,
  title = {Topics in Geometric Group Theory},
  author = {de la Harpe, Pierre},
  date = {2000},
  year = {2000},
  publisher = {University of Chicago Press},
  series = {Chicago Lectures in Mathematics},
  isbn = {978-0-226-31719-9}
}

\end{document}